\documentclass[reqno]{amsart}
\usepackage{amsmath,amsthm,amssymb,enumitem,times,xcolor,bm,esint,mathtools,dsfont}
\usepackage{leftidx}
\usepackage[mathscr]{eucal}
\usepackage{harpoon}

\usepackage{tikz}
\usepackage[titletoc]{appendix} %附录编号
\usepackage[%backref=page,
linktocpage=true,colorlinks,citecolor=magenta,linkcolor=blue,urlcolor=magenta]{hyperref}
\allowdisplaybreaks[4]
\providecommand{\U}[1]{\protect \rule{.1in}{.1in}}

\usepackage{geometry}
\newtheorem{theorem}{Theorem}[section] 
\newtheorem{corollary}[theorem]{Corollary} 
\newtheorem{definition}[theorem]{Definition} 
\newtheorem{lemma}[theorem]{Lemma} 
\newtheorem{proposition}[theorem]{Proposition}

\theoremstyle{remark}
\newtheorem{remark}[theorem]{Remark} 
\numberwithin{equation}{section}  
\def\CR{{\mathcal C}}

\def\PR{{\mathcal P}}
\def\QR{{\mathcal Q}}
\def\SR{{\mathcal S}}

\def\C{\mathbb C} %空心字母
\def\R{\mathbb R}
\def\N{\mathbb N}

\DeclareMathAlphabet{\mathpzc}{OT1}{pzc}{m}{it}
\usepackage{bm}

\def\al{\alpha}
\def\ga{\gamma}

\def\ep{\varepsilon}
\def\iy{\infty}
\def\la{\lambda}

\def\pa{\partial}
\def\ti{\tilde}

\def\ov{\overline}

\def\f{\frac}
\def\sbr#1{\left(#1\right)} %圆括号
\def\lbr#1{\left\{#1\right\}} %花括号
\def\abs#1{\left\lvert#1\right\rvert} %绝对值
\def\vs{\vskip .2in} %跳跃0.2in的行间距
\def\rd{\mathrm d}   %直立的微分符号d
\def\Im{\mathrm{Im\,}}  %直立的虚部分符号Im
\def\loc{{\operatorname{\rm loc}}} %直立的局部符号loc
\begin{document}
\theoremstyle{plain}

\title[Classification of Blow-up Solutions]{On the Classification of blow-up solutions of a singular Liouville equation on the disk}

\author{Zhijie Chen}\address{Department of Mathematical Sciences, Yau Mathematical Sciences Center, Tsinghua University, Beijing, China}\email{zjchen2016@tsinghua.edu.cn}
\author{Houwang Li}\address{School of Mathematics and Statistics, Beijing Institute of Technology, Beijing, China}\email{houwangli@bit.edu.cn}
\author{Tuoxin Li}\address{Department of Mathematics, The Chinese University of Hong Kong, HK, China}\email{txli@math.cuhk.edu.hk}
\author{Juncheng Wei}\address{Department of Mathematics, The Chinese University of Hong Kong, HK, China}\email{wei@math.cuhk.edu.hk}


%\thanks{...}

\begin{abstract}
    	We study the blow-up behavior of solutions to the singular Liouville equation
	\[
	\Delta \tilde u+\lambda e^{\tilde u}=4\pi\alpha\delta_0
	\quad\text{in }B,\quad
	\tilde u=0
	\quad\text{on }\partial B,
	\]
	where $\alpha>0$, $\lambda>0$ and $B\subset\mathbb R^2$ is the unit disk. Our main results give a complete classification of all blow-up solutions and determine the exact number of solutions to the above equation. More precisely, for fixed $\al>0$ and $\la\in(0,\la_\al)$, the singular Liouville equation has exactly $\lceil \alpha\rceil+2$ solutions (up to rotation): a unique minimal energy solution; a unique singular sequence blowing up at the origin; and for each $1\le m\le\lceil \alpha\rceil$, a unique $m$-peak sequence whose blow-up points are the vertices of a regular $m$-gon centered at the origin. This result answers the questions raised in Bartolucci-Montefusco \cite{Bartolucci-Montefusco06} and Bartolucci \cite{Bartolucci10}.  We also prove the non-degeneracy of these solutions. Thus we provide a full description of the blow-up structure for the singular Liouville equation on the disk.
\end{abstract}
\maketitle
%%%%%%%----------------------------------------------------------------------

\section{Introduction}\label{intro_class}

	In this paper, we study the singular Liouville equation 
    \begin{equation}\label{equ--1} \begin{cases}
			\Delta\ti u+\la e^{\ti u}=4\pi\al\delta_0, \quad\text{in}~B,\\
			\ti u=0,\quad\text{on}~\pa B,
		\end{cases}
	\end{equation}
	where $\al>0$, $\la>0$ and $B=\{x\in\R^2:~|x|^2< 1\}$ is the unit disk.
	Equation \eqref{equ--1} is a particular case of the following general singular Liouville equation
	\begin{equation}\label{0626-1}
		\left\{	\begin{aligned}
			&	\Delta\ti u+\la e^{\ti u}=4\pi \sum_{j=1}^n\ga_j\delta_{p_j}  \text{ in } \Omega,\\
			&  \ti u=0  \text{ on } \partial\Omega,
		\end{aligned}\right.		
	\end{equation}
	where $\la>0$, $\Omega$ is a smooth bounded domain in $\R^2$,  $p_1,\cdots,p_n$ are distinct points in $\Omega$, and $\ga_1,\cdots,\ga_n\ge0$.  
	In recent years, this problem and qualitatively similar ones have attracted great attention due to its profound connections with geometry and physics. For example, this type of equation arises in the study of Berger’s problem regarding the existence of metrics with prescribed Gaussian curvature. We refer to \cite{Aubin12, Chang-Yang88, Chang-Gursky-Yang93,Kazdan-Warner75} and the references therein for geometric applications. Furthermore, such an equation is introduced in astrophysics and combustion theory \cite{Chandrasekhar57, Gelfand63, Mignot-Puel-Murat79}, mean-field limit of vortices in Euler flows \cite{Caglioti-Lions-Marchioro-Pulvirenti92, Caglioti-Lions-Marchioro-Pulvirenti95, Chanillo-Kiessling94}, and vortices in the relativistic Maxwell-Chern-Simons-Higgs theory \cite{Caffarelli-Yang95, Chae-Imanuvilov00, Lin00, Struwe-Tarantello98}.
	
	The most important feature of equation \eqref{0626-1} is the blow-up phenomenon of solutions as $\la\to0$. Indeed, let $\la_n\to0$ and $\ti u_n$ be a solution sequence of \eqref{0626-1} with $\la=\la_n$.   
    We say that $\tilde u_n$ is a sequence of blow-up solutions of \eqref{0626-1} with parameter $\la_n\to0^+$ if, for some point $p\in\Omega$, there exists 
	$x_n\to p$ such that $\tilde u_n(x_n)+\log \la_n-2\ga_p\log|x_n-p|\to+\iy$, and there is a uniform bound on its mass
	\begin{equation*}
		\la_n\int_{\Omega}e^{\tilde u_n}\rd x\le C.
	\end{equation*}
	Here we set $\ga_p=0$ if $p$ is not a singular source.
	If a $\ga_j>0$ is not a positive integer ($\ga_j\in(0,\iy)\setminus\N^+$, $\N^+$ is the set of positive integers), we say $p_j$ is a {\it non-quantized singular source}, and if $\ga_j>0$ is a positive integer, we say $p_j$ is a {\it quantized singular source}.

	For a number of applications, considerable effort has been devoted to understanding precise blow-up behavior and the refined asymptotic behavior.
	If a blow-up point $p$ is either a regular point or a “non-quantized” singular source, the asymptotic behavior of $\tilde u_n$ around $p$ is well understood 
	\cite{BCLT04,Bartolucci-Jevnikar-Lee-Yang19-1,Bartolucci-Tarantello02,Bartolucci-Yang-Zhang24-1,Chen-Lin02,Chen-Lin10,Gluck,Li99,Li-Shafrir94,Zhang06,Zhang09}. As a matter
	of fact, $\tilde{u}_n$ performs ``simple'' blow-up around $p$ in the sense that, after scaling,
	the sequence $\tilde{u}_n$ behaves as a single bubble around the maximum point. 
	However, if $p$ happens to be a quantized singular source, the so-called ``non-simple'' blow-up phenomenon may occur (see \cite{Bartolucci-Tarantello,Kuo-Lin,Wei-Zhang-1,Wei-Zhang-2,Wei-Zhang-3}).
	Due to the analytically essential difference between $\ga_p\in\N^+$ and $\ga_p\not\in\N^+$, the study of non-simple blow-up solutions has been a major challenge for Liouville equations and progress on this problem remained limited for many years.
	Recently, D’Aprile, Wei and Zhang \cite{Teresa-Wei-Zhang25} made significant progress, demonstrating that the boundary data have a great influence on the profile of blow-up solutions inside. 
	Applied to the Dirichlet problem, this result shows that if $\tilde u_n$ is a sequence of blow-up solutions of \eqref{0626-1}, then $\tilde u_n$ is simple around any blow-up point in $\Omega$.
	
	Meanwhile, the existence of solutions with such blow-up phenomena has also attracted much attention. By the reduction method, many types of blow-up solutions have been constructed in \cite{Baraket-Pacard98, D’Aprile13, D’Aprile19, D’Aprile-Wei-Zhang24,Teresa-Wei20, delPino-1,del Pino-Kowalczyk-Musso05, Esposito05}, and we refer to their references for more related results.
		
	\vskip0.1in
	Despite the various interesting results mentioned above, we restrict ourselves to the simplest model problem: $\la_n\to0$ and $\ti u_n$ is a solution sequence of \eqref{equ--1} with $\la=\la_n$.
    Setting $u_n=\tilde u_n-2\al\log|x|$, one finds that $u_n$ satisfies
	\begin{equation}\label{equ-2} \begin{cases}
			\Delta u_n+\la_n|x|^{2\al}e^{u_n}=0, \quad\text{in}~B,\\
			u_n=0,\quad\text{on}~\pa B.
		\end{cases}
	\end{equation}
	Then, by the work of \cite{ Bartolucci-Tarantello02,Brezis-Merle,Esposito05,Li-Shafrir94,Nagasaki-Suzuki90}, we get that either $u_n\to0$ uniformly in $C^2(\bar B)$; or $u_n$ blows up, and there exists a non-empty set $\SR=\{a_1,\cdots,a_m\}\subset B$ such that
	\[\la_n|x|^{2\al}e^{u_n}\to\sum_{i=1}^m\beta_{a_i}\delta_{a_i} \]
	in the sense of measures, where 
	\[\beta_{a_i}=\begin{cases}8\pi,\quad\text{if}~a_i\neq 0,\\ 8\pi(1+\al),\quad\text{if}~a_i=0.\end{cases}\]
	Moreover, if $0\not\in\SR$, there holds that $(a_1,\cdots,a_m)\subset B^m\setminus\{a_i=a_j,~i\neq j\}$ is a critical point of the following function
	\begin{equation}\label{Phi} \Phi_m(a_1,\cdots,a_m)=\sum_{i=1}^m\log|a_i|^{2\al}+4\pi\sum_{i=1}^m H(a_i,a_i) +4\pi\sum_{i,j=1,i\neq j}^m G(a_i,a_j), 
	\end{equation}
	where $G(x,y)$ is the Green function of $\Delta$ with respect to the Dirichlet condition and $H(x,y)$ is its regular part.
	We say that a solution sequence $u_n$ is an {\it $m$-peak bubble solution} if it satisfies
                    \[\la_n|x|^{2\al}e^{u_n}\to8\pi\sum_{j=1}^m\delta_{a_j}, \]
    for distinct points $a_j\in B\setminus\{0\}$; and say that a solution sequence $u_n$ is a {\it singular bubble solution} if it satisfies
                    \[\la_n|x|^{2\al}e^{u_n}\to8\pi(1+\al)\delta_0. \]
    %By \cite[Theorem 6.2]{Chen-Lin02}, \cite[Theorem 1.1']{Chen-Lin10} and \cite[Theorem 1.4]{Li-Li-Wei-25}, 
    Since we focus on the unit disk, by using results in \cite{Chen-Lin02,Chen-Lin10,Li-Li-Wei-25}, we get
    \begin{equation}\label{energy-1}
        \la_n\int_B|x|^{2\al}e^{u_n}\rd x=\begin{cases} 8m\pi+O(\la_n^2), \quad&\text{if $u_n$ is an $m$-peak bubble solution},\\ 8\pi(1+\al)+O(\la_n^2), \quad&\text{if $u_n$ is a singular bubble solution.}\end{cases}
    \end{equation}

    Our aim is to classify all blow-up solutions of \eqref{equ-2}, and then count the total number of solutions to the following equation, which is equivalent to \eqref{equ--1}:
    \begin{equation}\label{equ-0} \begin{cases}
			\Delta u+\la |x|^{2\al}e^{u}=0, \quad\text{in}~B,\\
			 u=0,\quad\text{on}~\pa B.
		\end{cases}
	\end{equation}
    For fixed $\al>0$, we denote the solution set as
        \[\CR_\la=\{u\in H_0^1(B):~\text{$u$ is a solution of \eqref{equ-0}}\},\]
    and inspired by the above concentration behavior, we define for any $s\ge0$,
        \[\CR_{\la,s}=\lbr{u\in\CR_\la:~\abs{\la\int_B|x|^{2\al}e^u\rd x-8\pi s}<4\la}.\]
    Then by \eqref{energy-1}, it is easy to see that for $\la>0$ small enough, an $m$-peak bubble solution belongs to $\CR_{\la,m}$ and a singular bubble solution belongs to $\CR_{\la,1+\al}$.
	Moreover, by using the Pohozaev identity, we get that (see for example \cite[Proposition 5.7]{Bartolucci-Malchiodi13})
	\begin{equation}\label{bound} 
        \la\int_B|x|^{2\al}e^{u}\rd x<8\pi(1+\al),\quad\forall~u\in\CR_\la.
	\end{equation}
    Thus, for fixed $\al>0$ and $\la>0$ small enough, it holds
        \[\CR_{\la}=\CR_{\la,0}\bigcup
        \cup_{l=1}^{\lceil \alpha\rceil}\CR_{\la,l}\bigcup\CR_{\la,1+\al}.\]
    where $\lceil \alpha\rceil=\min\{n\in\mathbb{Z}: \alpha\le n\}$ . 
    By \cite{Esposito05}, we know that for $\la>0$ small enough, $\CR_{\la,0}$ consists of the unique minimal energy solution, so that $\sharp\CR_{\la,0}=1$. 
    In the following, we focus on the classification of blow-up solutions.
    
	\vskip0.1in
	
	To classify $m$-peak bubble solutions, we have to classify all critical points of $\Phi_m$. In the case of the unit disk, the functional $\Phi_m$ \eqref{Phi} can be written explicitly in complex notation. Identifying $(x_1,x_2)\in\mathbb R^2$ with $z=x_1+ix_2\in\mathbb C$, one obtains
	\begin{equation}\label{Phi complex}\begin{aligned}
	\Phi_m(z_1,\dots,z_m)
	=
	&2\alpha\sum_{j=1}^m \log|z_j|
	+2\sum_{j=1}^m \log(1-|z_j|^2)\\
	&+2\sum_{\substack{j,k=1\\ j\neq k}}^m \log|1-z_j\overline z_k|
	-2\sum_{\substack{j,k=1\\ j\neq k}}^m \log|z_j-z_k|.
	\end{aligned}\end{equation}
	This finite-dimensional problem is closely related to vortex-type Hamiltonians and is already of independent interest. It is motivated by the study of the point vortex model for two-dimensional Euler flows. We refer the interested readers to \cite{Boatto-Cabral03,Gueron-Shafrir99,Newton01,Palmore82} for more details concerning the $m$-vortex problem. However, even on the disk, the structure of $\Phi_m$ becomes rapidly intricate as $m$ increases. Direct computation is effective only in very low dimensions, and a more structural argument is needed in order to understand all critical configurations. Our first theorem shows that the situation is in fact completely rigid: every critical point of $\Phi_m$ is given by a regular polygon centered at the origin.
	\begin{theorem}\label{thm Phim_class}
		Let $m<1+\alpha$ be a positive integer and
			\begin{equation}\label{def r al m}
			r_{\alpha,m}=\left(\frac{\al+1-m}{\al+1+m}\right)^{\frac{1}{2m}}.    
		\end{equation}
		  Then $(z_1,\cdots, z_m)\in (B^*)^m\setminus \Theta_m$ is a critical point of $\Phi_m$ if and only if $z_1,\cdots, z_m$ are the vertices of a regular $m$-polygon satisfying
		\begin{equation*}
			|z_1|=\cdots=|z_m|=r_{\alpha,m}.
		\end{equation*}
        Moreover, $\Phi_{m'}$ has no critical points in $(B^*)^{m'}\setminus \Theta_{m'}$ for any integer $m'\ge1+\al$.
	\end{theorem}
	Here, as usual, we denote $B^*=B\setminus\{0\}$, and $\Theta_m:=\{(z_1,\cdots,z_m)\in B^m: z_j=z_l \text{ for some } j\neq l\}$. 
	When $m=1$ or $2$, Theorem \ref{thm Phim_class} has already been proved in \cite{Bartolucci-Montefusco06}. However, the structure of $\Phi_m$ becomes increasingly complicated when $m$ gets larger, and it is almost impossible to prove the result by direct computation. Indeed, we have to use some more indirect argument to prove it. See Section \ref{proof1.2_class} for details. By Theorem \ref{thm Phim_class}, it is easy to see that even for integer $\al$, $\CR_{\la,1+\al}$ contains only singular bubble solutions.

	Our main result is a full classification of the blow-up solutions of \eqref{equ-2}. It states that, for small $\lambda$, every blow-up solution is either the singular branch concentrating at the origin or one of the polygonal $m$-peak branches, and that no additional bubbling pattern can occur.

	\begin{theorem}\label{thm main_class}
		For any $\alpha>0$, there exists a small $\la_\alpha>0$ such that for $\la\in(0,\la_\al)$, the equation \eqref{equ-0} has exactly one sequence of singular bubble solutions that blows up only at $0$, and for each positive integer $m<\alpha+1$, the equation \eqref{equ-0} admits exactly one sequence of $m$-peak bubble solutions (up to rotation), whose blow-up points are the vertices of a regular $m$-polygon centered at the origin with radius $r_{\alpha,m}$ given by \eqref{def r al m}. 
        Moreover, for $\la\in(0,\la_\al)$, $\sharp\CR_{\la,1+\al}=1$ and $\sharp\CR_{\la,m}\mathbin{/}\sim=1$ for each $m=1,\cdots,\lceil \alpha\rceil$, where $\sim$ is the equivalence induced by the rotations.
	\end{theorem}
    
    In total, we get the exact number of solutions of \eqref{equ-0}.
    \begin{theorem}\label{thm number}
        For any $\alpha>0$, there exists a small $\la_\alpha>0$ such that for $\la\in(0,\la_\al)$, the number of solutions of \eqref{equ-0} is
            \[\sharp\CR_\la\mathbin{/}\sim=\lceil \alpha\rceil+2,\]
        where $\sim$ is the equivalence induced by the rotations.
    \end{theorem}

    By using ODE methods, we see that there are two radial solutions of \eqref{equ-0}
        \[u_{\la,j}(x)=\log\f{\Lambda_{\la,j}}{\sbr{1+\f{\Lambda_{\la,j}}{8(1+\al)^2}|x|^{2(1+\al)}}^2}-\log\la,    \quad\text{for}~j=1,2,\]
    where $0<\Lambda_{\la,1}<\Lambda_{\la,2}$ are two solutions  of
            \[\f{\Lambda}{\sbr{1+\f{\Lambda}{8(1+\al)^2}}^2}=\la.\]
    It is easy to see that
        \[\lim_{\la\to0}\Lambda_{\la,1}=0,\quad \lim_{\la\to0}\Lambda_{\la,2}=+\iy,\]
    so that for $\la$ small, $\CR_{\la,0}=\{u_{\la,1}\}$ consists of the unique minimal energy solution, and $\CR_{\la,1+\al}=\{u_{\la,2}\}$ consists of the unique singular bubble solution.

	Beyond classification,  we study the non-degeneracy of these blow-up solutions. As is well-known, non-degeneracy is closely related to local uniqueness, and it is a fundamental property in the analysis of partial differential equations that ensures the structural rigidity of solutions under the invariant transformations of the equation. We refer the readers to \cite{Bartolucci-Jevnikar-Lee-Yang18,Bartolucci-Jevnikar-Lin19,Bartolucci-Yang-Zhang24-1,delPino-2,OSS13} for relevant non-degeneracy results.
	Now we introduce our definition of non-degeneracy, based on the definition in \cite{DKM}. Let $u\in\CR_\la$ be a solution to \eqref{equ-0}.	Let $\phi$ be a bounded solution to the linearized equation
	\begin{equation}\label{eq phi}
		\begin{cases}
			\Delta  \phi+\lambda|x|^{2\alpha}e^{ u}\phi=0, \quad\text{in}~B,\\
			\phi=0, \text{ on } \partial B.
		\end{cases}
	\end{equation}
	Then the rotational invariance of \eqref{equ-0} is reflected in the solution set of \eqref{eq phi}. 
	\begin{definition} We have the following definition of non-degeneracy:
		\begin{enumerate}[label=(\roman*)]
			\item For an $m$-peak bubble solution $u\in\CR_{\la,m}$, we say that it is non-degenerate if the solution to \eqref{eq phi} has the form $\phi=c(x_1\frac{\partial u}{\partial x_2}-x_2\frac{\partial u}{\partial x_1})$ for some constant $c$.
			\item For a radial solution $u\in\CR_{\la,0}\cup\CR_{\la,1+\al}$,  we say that it is non-degenerate if the only solution to \eqref{eq phi} is $\phi\equiv 0$.
		\end{enumerate}
	\end{definition}
	
	Then our result can be stated as follows:
	\begin{theorem}\label{thm nondegen_class}
	For any $\alpha>0$, there exists a small $\la_\alpha>0$ such that for $\la\in(0,\la_\al)$, each solution $u\in\CR_{\la}$ is non-degenerate.
	\end{theorem}

	The proof of our main result, Theorem \ref{thm main_class}, has two principal ingredients. The first is contained in Theorem \ref{thm Phim_class}, that is, the classification of the critical points of the reduced Hamiltonian. 	
	This classification problem was initiated by  Bartolucci and Montefusco in \cite{Bartolucci-Montefusco06}, in which the authors characterized the $1$-peak and $2$-peak bubble solutions. Later Bartolucci \cite{Bartolucci10} studied the $m$-peak bubble solutions, under the assumptions that the blow-up points are the vertices of a regular $m$-polygon. Now our approach is to recast the critical-point system into a polynomial framework and exploit the special algebraic structure. This reduction makes it possible to rule out all asymmetric configurations and to identify the regular $m$-polygon as the unique admissible pattern. 
	
	The second ingredient is local uniqueness.  Once the candidate blow-up configurations have been identified, one still has to show that each admissible pattern gives rise to only one blow-up branch modulo rotation. This part of the argument is treated separately according to the nature of the blow-up. For the $m$-peak solutions, we analyze the structure of the Hessian of the reduced Hamiltonian, whose only degeneracy comes from rotational invariance. For the singular blow-up solutions, the argument for the non-quantized case $\alpha\notin\mathbb N$ and the quantized case $\alpha\in\mathbb N$ is also slightly different. For more results about the local uniqueness of Liouville equations or mean field equations of various types, we refer to \cite{Bartolucci-Jevnikar-Lee-Yang19-1,Bartolucci-Jevnikar-Lee-Yang19-2,Bartolucci-Jevnikar-Lee-Yang20,Bartolucci-Yang-Zhang24-2,Wu-Zhang21}.
	
	The combination of these two steps yields the complete classification of blow-up solutions. The non-degeneracy result then follows from a further study of the linearized operator around each branch, with the same spirit of the proof of the local uniqueness result.
	
	The paper is organized as follows. In Section \ref{proof1.2_class} we classify the critical points of the reduced Hamiltonian $\Phi_m$ and prove Theorem \ref{thm Phim_class}. Section \ref{sec unique} is devoted to local uniqueness and is divided into three subsections, corresponding respectively to the $m$-peak case,  the singular non-quantized case $\alpha\notin\mathbb N$ and the singular quantized case $\alpha\in\mathbb N$. In Section \ref{proof1.1_class}, we complete the proof of Theorem \ref{thm main_class}. In Section \ref{proof1.3_class} we prove the non-degeneracy of the blow-up solutions. Finally, the appendix contains the matrix computations needed for the analysis of the Hessian in the $m$-peak case.

\vskip0.2in

\section{Locating the blow-up set}\label{proof1.2_class}

This section is devoted to the classification of the critical points of $\Phi_m$. Instead of working directly with the nonlinear critical-point system \eqref{eq zj}, we reformulate the problem in terms of two auxiliary polynomials $P$ and $Q$ and derive an identity that strongly constrains their roots. This algebraic viewpoint reveals the rigidity of the admissible configurations and ultimately shows that every critical point corresponds to a regular $m$-polygon centered at the origin.

First, by \eqref{Phi complex} and direct computation, we have
\begin{equation*}
\begin{aligned}
    \pa_{z_j}\Phi_m (z_1,\cdots,z_m)
    =&\al\f{\bar z_j}{|z_j|^2}-2\f{\bar z_j}{1-|z_j|^2}+2\sum_{k=1,k\neq j}^m\f{\bar z_k(\bar z_jz_k-1)}{|z_j\bar z_k-1|^2}-2\sum_{k=1,k\neq j}^m\f{\bar z_j-\bar z_k}{|z_j-z_k|^2} \\
    =&\al\f{1}{z_j}+2\f{1}{z_j-\bar z_j^{-1}}+2\sum_{k=1,k\neq j}^m\f{1}{z_j-\bar z_k^{-1}}-2\sum_{k=1,k\neq j}^m\f{1}{z_j-z_k}.
\end{aligned}   
\end{equation*}
Then, since $\pa_{\bar z_j}\Phi_m=\ov{\pa_{z_j}\Phi_m}$, we get that if $(z_1,\cdots,z_m)\in (B^*)^m\setminus \Theta_m$ is a critical point of $\Phi_m$, then
\begin{equation}\label{eq zj} 
    \sum_{k=1,k\neq j}^m\f{1}{z_j-z_k}=\f{\al}{2}\f{1}{z_j}+\sum_{k=1}^m\f{1}{z_j-\bar z_k^{-1}},\quad\forall~1\le j\le m.
\end{equation}

To prove Theorem \ref{thm Phim_class}, it is useful to introduce the polynomials
\begin{equation}\label{def P Q}
    P(z)=\prod_{j=1}^m(z-z_j),\quad Q(z)=\prod_{j=1}^m(z-\bar z_j^{-1}).
\end{equation}
Then it is enough to show that
\begin{equation*}
    P(z)=z^m-e^{i\theta}r_{\alpha,m}^m
\end{equation*}
for some $\theta\in\R$.
This idea comes from Appendix A in \cite{Bartolucci-Tarantello} and Lemma 2 in \cite{Kuo-Lin}. More specifically, the authors showed that after scaling, the local maximum points of a non-simple blow-up sequence of solutions are evenly distributed on the unit circle by studying the corresponding polynomial. In our case, though we cannot obtain the result directly as in \cite{Bartolucci-Tarantello} or \cite{Kuo-Lin}, we have the following useful relation between $P(z)$ and $Q(z)$:

\begin{lemma}\label{lem P Q}
Let $P(z)$ and $Q(z)$ be given in \eqref{def P Q}. Then \eqref{eq zj} implies
    \begin{equation}\label{eq P Q}
    2zP'Q'+(\al P'-zP'')Q-((\al+2)Q'+zQ'')P=0,\quad\forall~z\in\C.
\end{equation}
\end{lemma}
\begin{proof}
    Taking derivatives of the rational function $\f{P(z)}{Q(z)}$, we obtain
    \[\sbr{\f{P(z)}{Q(z)}}'\Bigg|_{z=z_j}=\f{\prod_{k=1,k\neq j}^m(z_j-z_k)}{\prod_{k=1}^m(z_j-\bar z_k^{-1})},\]
and 
    \[\sbr{\f{P(z)}{Q(z)}}''\Bigg|_{z=z_j}=2\sbr{\f{P(z)}{Q(z)}}'\Bigg|_{z=z_j}\sbr{\sum_{k=1,k\neq j}^m\f{1}{z_j-z_k}-\sum_{k=1}^m\f{1}{z_j-\bar z_k^{-1}} }.\]
Then by using \eqref{eq zj}, we get that
    \[\alpha\sbr{\f{P(z)}{Q(z)}}'\Bigg|_{z=z_j}-z_j\sbr{\f{P(z)}{Q(z)}}''\Bigg|_{z=z_j}=0,\quad\forall~1\le j\le m.\]
In other words, 
\begin{equation}\label{zj zero}
    z_1,\cdots,z_m \text{ are distinct zeros of } 2zP'Q'+\al P'Q-zP''Q.
\end{equation}

In the same way, we have
 \[(\al+2)\sbr{\f{Q(z)}{P(z)}}'\Bigg|_{z=\bar z_j^{-1}}+\bar z_j^{-1}\sbr{\f{Q(z)}{P(z)}}''\Bigg|_{z=\bar z_j^{-1}}=0,\quad\forall~1\le j\le m.\]
and consequently,
\begin{equation}\label{bar zj zero}
\bar z_1^{-1},\cdots,\bar z_m^{-1} \text{ are distinct zeros of } 2zP'Q'-(\al+2) PQ'-zPQ''. 
\end{equation}

In total, we get from \eqref{zj zero} and \eqref{bar zj zero} that $z_1,\cdots,z_m,\bar z_1^{-1},\cdots,\bar z_m^{-1}$ are $2m$ distinct zeros of the polynomial
    \begin{equation*}
        2zP'Q'+(\al P'-zP'')Q-((\al+2)Q'+zQ'')P,
    \end{equation*}
    whose degree is at most $2m-1$. Therefore the polynomial must be identically zero, and the lemma is proved.
\end{proof}

\begin{proposition}\label{prop no root on S}
   If \eqref{eq zj} holds, then $P(z)$ has no root on the unit circle.
\end{proposition}
\begin{proof}
    Suppose on the contrary that $P(z)$ has a zero $z_j$ such that $|z_j|=1$. Let $s$ be the multiplicity of $z_j$. We first show that $s\ge 2$. In fact, since $z_j=\bar z_j^{-1}$, we see that $z_j$ is also a root of $Q(z)$ of multiplicity $s$. By \eqref{eq P Q}, we have $P'(z_j)Q'(z_j)=0$, then either $P'(z_j)=0$ or $Q'(z_j)=0$ implies that $z_j$ is a repeated root, that is, $s\ge 2$.

    Now let $P(z)=c_s(z-z_j)^s+...$ and $Q(z)=c_s'(z-z_j)^s+...$ be the Taylor expansions of $P$ and $Q$ near $z_j$ respectively. Then we expand the left hand side of \eqref{eq P Q} near $z_j$, and it follows that the coefficient of the leading term $(z-z_j)^{2s-1}$ is $c_sc_s'(\alpha s-(\alpha+2)s)\neq 0$, which is a contradiction.
\end{proof}

Next we give the classification of $P(z)$ under the condition \eqref{eq zj}. In what follows we write $\beta=\alpha+1$ for convenience. Let $z_j=e^{-2i\zeta_j}$ for $1\le j\le m$, where $\zeta_j\in\C$ is understood in the sense of mod $\pi$.  Using the identity $1-e^{-2ix}=2ie^{-ix}\sin x$, we find that \eqref{eq zj} becomes
	\begin{equation}\label{eq zetaj}
		\sum_{k=1, k\neq j}\cot(\zeta_j-\zeta_k)-\sum_{k=1}^m\cot(\zeta_j-\bar\zeta_k)=-i\beta,\quad\forall~1\le j\le m
	\end{equation}
	Let $\tau(\zeta)=\prod_{j=1}^m\sin(\zeta-\zeta_j)$ and $\tau_\sharp(\zeta)=\prod_{j=1}^m\sin(\zeta-\bar\zeta_j)$. Then
	\begin{equation}\label{P tau}
		P(e^{-2i\zeta})=(-2i)^m  e^{-im\zeta}e^{-i\sum_{j=1}^m\zeta_j}\tau(\zeta)
	\end{equation}
	Define
	\begin{equation}\label{psi 1}
		\psi_1(\zeta)=e^{-i\beta\zeta}\frac{\tau_\sharp(\zeta)}{\tau(\zeta)}.
	\end{equation}
	Then we have
	\begin{equation}\label{f zeta}
		f(\zeta):=\frac{\psi_1'(\zeta)}{\psi_1(\zeta)}=-i\beta-\sum_{j=1}^m\cot(\zeta-\zeta_j)+\sum_{j=1}^m\cot(\zeta-\bar\zeta_j).
	\end{equation}
	We observe that $f$ has simple poles at $\zeta_j$ and $\bar\zeta_j$. From \eqref{eq zetaj}, one immediately obtains that
	\begin{equation}\label{res f}
		\text{Res}f\Big|_{\zeta=\zeta_j}=-1,\quad  \text{Res}f^2\Big|_{\zeta=\zeta_j}=0.
	\end{equation}
	By symmetry it also holds that
	\begin{equation}\label{res f1}
		\text{Res}f\Big|_{\zeta=\bar\zeta_j}=1,\quad \text{Res}f^2\Big|_{\zeta=\bar\zeta_j}=0.
	\end{equation}
	
	Now we define $g(\zeta)=f'(\zeta)+f^2(\zeta)+\beta^2$. Then by definition, we see that $\psi_1$ satisfies the equation
	\begin{equation}\label{eq psi}
		-\psi''+g(\zeta)\psi=\beta^2\psi.
	\end{equation}
By \eqref{res f}, \eqref{res f1} and Lemma 2 in \cite{Veselov01}, the operator $-\frac{\rd^2}{\rd\zeta^2}+g(\zeta)$ is monodromy-free, in the sense that the corresponding Schr\"odinger equation
    \begin{equation*}
		-\psi''+g(\zeta)\psi=\lambda\psi
	\end{equation*}
    has all solutions meromorphic for all values of $\lambda$. We refer to \cite{Hemery-Veselov14} and the references therein for related topics and applications. Then by Theorem 4.3 in \cite{Chalykh-Feigin-Veselov99} (see also \cite{Berest97,Berest-Lutsenko97}), there exist integers $1\le k_1<\cdots<k_l$ and phases $\phi_j$, such that
	\begin{equation}\label{g=W}
		g(\zeta)=-2\frac{\rd^2}{\rd\zeta^2}W(\sin(k_1\zeta+\phi_1),\cdots, \sin(k_l\zeta+\phi_l)),
	\end{equation}
	where $W$ represents the Wronskian, and
	\begin{equation}\label{psi HV}
		\psi_{2}:=\frac{W(\sin(k_1\zeta+\phi_1),\cdots, \sin(k_l\zeta+\phi_l),e^{-i\beta\zeta})}{W(\sin(k_1\zeta+\phi_1),\cdots, \sin(k_l\zeta+\phi_l))}
	\end{equation}
	is also a solution to \eqref{eq psi}. For convenience we use $W$ and $\hat{W}$ to denote the denominator and the numerator in \eqref{psi HV} respectively. Then from \eqref{eq psi} we see that the Wronskian $W(\psi_1,\psi_2)$ is constant. Since both $\psi_1$ and $\psi_{2}$ behave like  a constant multiple of $e^{-i\beta\zeta}$ as $\Im\zeta\to -\infty$, we must have $W(\psi_1,\psi_2)=0$, which implies
	\begin{equation}\label{psi B psi HV}
		\psi_{1}=C_0\psi_2
	\end{equation}
	for some $C_0\in \C$.
    
	On the other hand, by \eqref{eq zetaj} and \eqref{f zeta}, direct computation yields that
	\begin{equation*}
		{ g(\zeta)=2\sum_{k=1}^m\csc^2(\zeta-\zeta_k)}=-2(\log(\tau))'',
	\end{equation*}
	which, together with \eqref{g=W}, implies that
	\begin{equation}\label{W=tau}
		W(\zeta)=C_1e^{ia\zeta}\tau(\zeta)
	\end{equation}
	for some $a$ and nonzero $C_1\in\C$. Now we compare the frequency of $\zeta$ in \eqref{W=tau}. First we introduce a useful formula:
	\begin{equation}\label{W vandermonde}
		W(e^{ia_1\zeta},\cdots,e^{ia_l\zeta})=i^{\frac{l(l-1)}{2}}\prod_{1\le p<q\le l}(a_q-a_p)e^{i\sum_{k=1}^la_k\zeta},
	\end{equation}
	which is an easy consequence of Vandermonde determinant. Next, we use an $l$-dimensional vector $\delta\in\{\pm 1\}^l$ to indicate the choice of signs, and for each $\delta$, we define
	\begin{equation}
		s_\delta=\sum_{j=1}^l\delta_jk_j,\quad V_\delta=\prod_{1\le p<q\le l}(\delta_q k_q-\delta_pk_p).
	\end{equation} 
	  Then since $\sin\zeta=\frac{e^{i\zeta}-e^{-i\zeta}}{2i}$, we can write
	\begin{equation}\label{W expan}
		W(\zeta)=\sum_{\delta\in\{\pm 1\}^l}C_\delta e^{is_\delta\zeta},
	\end{equation}
	where 
	\begin{equation}\label{C del}
		C_\delta=(2i)^li^{\frac{l(l-1)}{2}}(\prod_{j=1}^l\delta_j e^{i\delta_j\phi_j})V_\delta.
	\end{equation}
	Similarly, we also have
	\begin{equation}\label{hat W expan}
		\hat{W}(\zeta)=\sum_{\delta\in\{\pm 1\}^l}D_\delta e^{i(s_\delta-\beta)\zeta},
	\end{equation}
	where
	\begin{equation}\label{D del}
		D_\delta=(-i)^{-l}C_\delta\prod_{j=1}^l(\beta+\delta_j k_j).
	\end{equation}

	From \eqref{W expan} we see that the extreme frequencies of $W$ are $\pm (\sum_{j=1}^l k_j)$. By definition, it is also easy to see that the extreme frequencies of $\tau$ are $\pm m$. Therefore \eqref{W=tau} forces that
	\begin{equation}
		\sum_{j=1}^l k_j=m,\quad a=0.
	\end{equation}
	As a result, we have
	\begin{equation}\label{tau expan}
		\tau(\zeta)=C_1^{-1}W(\zeta)=C_1^{-1}\sum_{\delta\in\{\pm 1\}^l}C_\delta e^{is_\delta\zeta}.
	\end{equation}
	Consequently,
	\begin{equation}\label{tau_sh expan}
		\tau_\sharp(\zeta)=\overline{\tau(\bar\zeta)}=\bar C_1^{-1}\sum_{\delta\in\{\pm 1\}^l}\overline{C_{-\delta}} e^{is_\delta\zeta}.
	\end{equation}
	Now collecting \eqref{psi 1}, \eqref{psi HV}, \eqref{psi B psi HV} and \eqref{tau expan}, we get that
	\begin{equation*}
		\hat{W}=C_0C_1e^{-i\beta\zeta}\tau_\sharp.
	\end{equation*}
	Therefore by \eqref{hat W expan} and \eqref{tau_sh expan}, 
	\begin{equation}\label{id}
		\sum_{\delta\in\{\pm 1\}^l}D_\delta e^{is_\delta\zeta}=C_2\sum_{\delta\in\{\pm 1\}^l}\overline{C_{-\delta}} e^{is_\delta\zeta},
	\end{equation}
	where $C_2=C_0C_1\bar C_1^{-1}$.

    Based on the preparation above, we can compute $\Im\phi_j$ for $1\le j\le l$ explicitly.
	\begin{proposition}\label{prop Im phi_j}
		For $\phi_j$ introduced in \eqref{g=W}, we have
		\begin{equation}\label{Im phi_j}
			\Im\phi_j=\frac{1}{4}\log\frac{\beta+k_j}{\beta-k_j},\quad
            1\le j\le l.
		\end{equation}
	\end{proposition}
	\begin{proof}
		The idea of the proof is to compare the frequency of $\zeta$ in \eqref{id} from the top to the bottom. We notice that though different $\delta\in\{\pm 1\}^l$ may lead to the same $s_\delta$, this cannot happen for the top three frequencies $m$, $(m-2k_1)$ and $(m-2k_2)$. This is the starting point of the proof.
		
		To be more precise, let
		\begin{equation*}
			\delta_+=(1,\cdots,1),\ \delta_-=(-1,\cdots,-1),\
			\delta_+'=(-1,1,\cdots,1), \
			\delta_-'=(1,-1,\cdots,-1). 
		\end{equation*}
	Then $s_{\delta_\pm}=\pm m$ and $s_{\delta_\pm'}=\pm (m-2k_1)$. Now by comparing the coefficients of $e^{\pm im\zeta }$, we get that
	\begin{equation*}
		\frac{D_{\delta_+}}{D_{\delta_-}}=\frac{C_{\delta_-}}{C_{\delta_+}}.
	\end{equation*}
Then it follows from \eqref{C del} and \eqref{D del} that
	\begin{equation}\label{k0 ratio}
		\prod_{j=1}^l(\frac{\beta-k_j}{\beta+k_j})=\Big|\frac{C_{\delta_+}}{C_{\delta_-}}\Big|^2=\Big|\prod_{j=1}^l e^{4i\phi_j}\Big|
	\end{equation}
	
	If $l=1$, then we are already done. If $l\ge 2$, then in the same way, by comparing the coefficients of $e^{\pm i(m-2k_1)\zeta }$, we can obtain
		\begin{equation}\label{k1 ratio}
		\frac{\beta+k_1}{\beta-k_1}\prod_{j=2}^l(\frac{\beta-k_j}{\beta+k_j})=\Big|\frac{C_{\delta_+}'}{C_{\delta_-}'}\Big|^2=\Big|e^{-4i\phi_1}\prod_{j=2}^l e^{4i\phi_j}\Big|,
	\end{equation}
which, together with \eqref{k0 ratio}, implies that $\Im\phi_1=\frac{1}{4}\log\frac{\beta+k_1}{\beta-k_1}$.
	
	Repeating the above steps gives $\Im\phi_2=\frac{1}{4}\log\frac{\beta+k_2}{\beta-k_2}$. However, when we proceed, we find that the next frequency can be either $(m-2k_3)$ or $m-(2k_1+2k_2)$, or they could be equal. In either case, when we write down the equality of coefficients as before, we immediately find that the terms involving $\phi_1$ or $\phi_2$ are automatically satisfied. Thus we can always get $\Im\phi_3=\frac{1}{4}\log\frac{\beta+k_3}{\beta-k_3}$ as expected. Repeating this process, we eventually obtain \eqref{Im phi_j} for all $j$.
	\end{proof}

\begin{remark}
    We have no constraint on $\Re\phi_j$ for all $1\le j\le l$.
\end{remark}
        
		Now we can go back to study the expression of $\tau(\zeta)$ and $P(z)$. If $l=1$, then $\Im\phi_1=\frac{1}{4}\log\frac{\beta+m}{\beta-m}$, and
		\begin{equation*}
			\tau(\zeta)=C_1^{-1}\sin(m\zeta+\phi_1)=\frac{e^{i(m\zeta+\phi_1)}-e^{-i(m\zeta+\phi_1)}}{2iC_1}.
		\end{equation*}
		By \eqref{P tau} and the fact that $P$ is monic, we must have
		\begin{equation}\label{P l=1}
			P(z)=z^m-e^{2i\phi_1}=z^m-e^{2i\Re\phi_1}\sqrt{\frac{\beta-m}{\beta+m}}.
		\end{equation}
		
		To complete the classification, we need to rule out all other cases. The idea is to use a continuity argument and focus on $P_\infty(z)$, the limit of $P(z)$ as $\beta\to\infty$.  

\begin{proposition}\label{prop |z|>=1}
		Let $m$ be a fixed positive integer. If $P_\infty(z)$ has at least one root that is not on the unit circle, then for any $\beta>m$, $P(z)$ has at least one root that is outside the unit circle.
\end{proposition}
	\begin{proof}
	    Since $P(z)$ is monic, it is easy to see that  its constant term is
	\begin{equation*}
		P(0)=\frac{C_{\delta_+}}{C_{\delta_-}}=(-1)^{\frac{l(l+1)}{2}}\prod_{j=1}^l e^{2i\phi_j}.
	\end{equation*}
By Proposition \ref{prop Im phi_j}, $\lim\limits_{\beta\to\infty}\Im\phi_j=0$, which implies $|P_{\infty}(0)| =1$. We denote the roots of $P_\infty(z)$ by $z_{\infty,j}$. Then $\prod_{j=1}^m|z_{\infty,j}|=1$, and the condition of the proposition implies some of $|z_{\infty,j}|>1$. By
Proposition \ref{prop no root on S} and continuity, we always have $|z_j|>1$ for all $\beta$. This completes the proof.
	\end{proof}

By Proposition \ref{prop |z|>=1}, it remains to consider the case when all roots of $P_\infty(z)$ are on the unit circle, that is, $|z_{\infty,j}|=1$ for all $j$. We first assume that $z_{\infty,j}$ are distinct. Now we go back to analyze \eqref{eq zj} in polar coordinates. Let $z_{j}=r_je^{i\theta_j}$. Then after some computation, \eqref{eq zj} reduces to
\begin{equation}\label{eq zj polar} 
    \sum_{k=1,k\neq j}^m\left(\f{1}{r_j-r_ke^{i(\theta_k-\theta_j)}}-\f{1}{r_j-r_k^{-1}e^{i(\theta_k-\theta_j)}}\right)=\f{\beta-1}{2}\f{1}{r_j}+\f{r_j}{r_j^2-1},\quad\forall~1\le j\le m.
\end{equation}
Since $z_{\infty,j}$ are distinct, by implicit function theorem, we see that $z_j$ is smooth in $\frac{1}{\beta}$ for $\beta$ large enough. Thus we may write
\begin{equation*}
    r_j=1+\frac{a_{1j}}{\beta}+\frac{a_{2j}}{\beta^2}+O(\frac{1}{\beta^3}).
\end{equation*}
Then the right hand side of \eqref{eq zj polar} gives
\begin{equation*}
    \frac{(a_{1j}+1)}{2a_{1j}}\beta-\frac{a_{1j}^2+2a_{1j}+2a_{2j}}{4a_{1j}^2}+O(\frac{1}{\beta}).
\end{equation*}
On the other hand, since $\theta_{\infty,j}$ are distinct, the left hand side of \eqref{eq zj polar} is $O(\frac{1}{\beta})$. Hence $a_{1j}=-1, a_{2j}=\frac{1}{2}$, that is,
\begin{equation}\label{rj exp}
    r_j=1-\frac{1}{\beta}+\frac{1}{2\beta^2}+O(\frac{1}{\beta^3}) ,\quad\forall~1\le j\le m.
\end{equation}

Next we consider the angles $\theta_j$. To this end, we go back to the functional $\Phi_m$ and take the partial derivative with respect to $\theta_j$ to obtain
\[\begin{aligned}
    \frac{\partial\Phi_m}{\partial\theta_j}=4r_j\sum_{k\neq j}r_k\sin(\theta_j-\theta_k)\Bigg(&\frac{1}{1+r_j^2r_k^2-2r_jr_k\cos(\theta_j-\theta_k)} \\ 
    & \quad -\frac{1}{r_j^2+r_k^2-2r_jr_k\cos(\theta_j-\theta_k)}\Bigg).
\end{aligned}\]
So
\begin{equation}\label{eq thetaj}
    \sum_{k\neq j}\frac{(1-r_k^2)r_k\sin(\theta_j-\theta_k)}{\left(1+r_j^2r_k^2-2r_jr_k\cos(\theta_j-\theta_k)\right)\left(r_j^2+r_k^2-2r_jr_k\cos(\theta_j-\theta_k)\right)}=0.
\end{equation}
Since $\theta_j=\theta_{\infty,j}+O(\frac{1}{\beta})$, by \eqref{rj exp}, the leading term of \eqref{eq thetaj} gives that
\begin{equation*}
    \sum_{k\neq j}\frac{\sin(\theta_{\infty,j}-\theta_{\infty,k})}{(1-\cos(\theta_{\infty,j}-\theta_{\infty,k}))^2}=0,
\end{equation*}
or equivalently,
\begin{equation*}
    \sum_{k\neq j}\frac{\cos \frac{\theta_{\infty,j}-\theta_{\infty,k}}{2}}{\sin^3\frac{\theta_{\infty,j}-\theta_{\infty,k}}{2}}=0.
\end{equation*}

The above discussion implies that $(\theta_{\infty,1},\cdots,\theta_{\infty,m})$ is a critical point of the functional $E_0$ defined by
\begin{equation*}
    E_0(\varphi_1,\cdots,\varphi_m):=\sum_{1\le j<k\le m}\csc^2\frac{\varphi_j-\varphi_k}{2}.
\end{equation*}
Let $f_0(\varphi)=\frac{2+\cos\varphi}{2\sin^4\frac{\varphi}{2}}$. Then $f_0(\varphi)\ge\frac{1}{2}$, and direct computation gives that
\begin{equation*}
    \frac{\partial^2 E_0}{\partial\varphi_j\partial\varphi_k}=-f_0(\varphi_j-\varphi_k).
\end{equation*}
From this we can show that $E_0$ is convex in the convex set
\begin{equation*}
    X:=\{(\varphi_1,\cdots,\varphi_m):  \varphi_1<\varphi_2<\cdots<\varphi_m< \varphi_1+2\pi, \sum_{j=1}^m\varphi_j=0\}.
\end{equation*}
In fact, for any $v\in\R^m$,
\begin{equation*}
    v^{\text{T}} D^2E_0 v=\sum_{j<k}f_0(\varphi_j-\varphi_k)(v_j-v_k)^2\ge \frac{1}{2}(v_j-v_k)^2=\frac{1}{2}(m\sum_{j=1}^m v_j^2-(\sum_{j=1}^m v_j)^2).
\end{equation*}
Therefore
    $v^{\text{T}} D^2E_0 v \ge \frac{m}{2}\|v\|^2$
if $\sum_{j=1}^m v_j=0$, and the convexity is proved. On the other hand, it is easy to check that for  $\varphi_j=\frac{2\pi}{m}(j-\frac{m+1}{2})$, $(\varphi_1,\cdots,\varphi_m)$ is a critical point of $E_0$ in $X$. By convexity, this is also the unique one.

To summarize, we have proved that after relabeling, $\theta_{\infty,j}$ must have the form
\begin{equation*}
    \theta_{\infty,j}=\theta_0+\frac{2\pi}{m}(j-1), \quad j=1,\cdots,m.
\end{equation*}
This implies $P_{\infty}(z)=z^m-e^{im\theta_0}$. This corresponds to the case $l=1$ in previous discussion, and $P(z)$ must have the form \eqref{P l=1}.

\vskip0.1in
Finally, we consider the case when $P_{\infty}(z)$ has a repeated root on the unit circle. Let $z_{\infty,1}=\cdots=z_{\infty,p}=e^{i\theta_{\infty}}$ be a root of $P_\infty$ with multiplicity $p$. In this case, $z_j$ may be no longer smooth in $\frac{1}{\beta}$, and we have the following result in algebraic geometry:
	
	 By Puiseux parametrization and the conjugate-factor factorization of a local Weierstrass polynomial, each local branch around $e^{i\theta_\infty}$ admits a parametrization $z_j=\phi(\gamma)$, where $\gamma$ is one of the complex $q$-root of $\frac{1}{\beta}$ for some $q\le p$. See, for example, Theorem 3.3 and Proposition 3.4 in \cite{Greuel-Lossen-Shustin}.
	 
 Using the above result, we see that if the leading order in the expansion of $z_j-1$ is not $\frac{1}{\beta}$ for some $j$, then  $z_{\infty,j}$ must split like a regular polygon locally. Since $|z_{\infty,j}|=1$, at least one local branch is outside the unit disk, which is a contradiction. Therefore, in what follows we may assume that we can expand $z_j=r_j^{i\theta_j}$ in the form 
	\begin{equation*}
		r_j=1+\frac{a_{j}}{\beta}+O(\frac{1}{\beta^{1+\frac{1}{p}}}),\quad \theta_j=\theta_{\infty,j}+\frac{b_{j}}{\beta}+O(\frac{1}{\beta^{1+\frac{1}{p}}}).
	\end{equation*}
Then by comparing the coefficients of $\beta$ in \eqref{eq zj polar} for $1\le j\le p$, we obtain
\begin{equation}\label{eq a b}
    \sum_{k=1,k\neq j}^p\left(\frac{1}{a_j-a_k-i(b_k-b_j)}-\frac{1}{a_j+a_k-i(b_k-b_j)}\right)=\frac{a_j+1}{2a_j},\ 1\le j\le p.
\end{equation}
Let $w_j=a_j+ib_j$. Then \eqref{eq a b} becomes
\begin{equation}\label{eq wj}
    \sum_{k=1,k\neq j}^p\frac{1}{w_j-w_k}-\sum_{k=1}^p\frac{1}{w_j+\bar w_k}-\frac{1}{2}=0,\ 1\le j\le p.
\end{equation}

This system is similar to \eqref{eq zj}, so we use the same idea to study it. To be more precise, we define
\begin{equation}\label{def PR QR}
    \PR(z)=\prod_{j=1}^p(z-w_j),\quad \QR(z)=\prod_{j=1}^p(z+\bar w_j).
\end{equation}
Then $\PR(z)=(-1)^p\overline{\QR(-\bar z)}$, and we have the following result similar to Lemma \ref{lem P Q}:
\begin{lemma}\label{lem PR QR}
Let $\PR(z)$ and $\QR(z)$ be given in \eqref{def PR QR}. Then
    \begin{equation}\label{eq PR QR}
   (\PR''-\PR')\QR+(\QR''+\QR')\PR- 2\PR'\QR'=0,\quad\forall~z\in\C.
\end{equation}
\end{lemma}
\begin{proof}
    The proof is similar to that of Lemma \ref{lem P Q}. We observe that \eqref{eq wj} is equivalent to
    \begin{equation*}
        \frac{1}{2}\sbr{\f{\PR''(z)}{\PR'(z)}}\Bigg|_{z=w_j}-\sbr{\f{\QR'(z)}{\QR(z)}}\Bigg|_{z=w_j}=\frac{1}{2},\quad\forall~1\le j\le p.
    \end{equation*}
    So
    \begin{equation*}
    w_1,\cdots,w_p \text{ are distinct zeros of } \PR''\QR-2\PR'\QR'-\PR'\QR.
\end{equation*}
In the same way, we can show that
 \begin{equation*}
    -\bar w_1,\cdots,-\bar w_p \text{ are distinct zeros of } \QR''\PR-2\PR'\QR'+\QR'\PR.
\end{equation*}

In total, we get that $w_1,\cdots,w_p,-\bar w_1,\cdots,-\bar w_p$ are $2p$ distinct zeros of the polynomial
    \begin{equation*}
       (\PR''-\PR')\QR+(\QR''+\QR')\PR- 2\PR'\QR',
    \end{equation*}
    whose degree is at most $2p-2$. Therefore the polynomial must be identically zero, and the lemma is proved.
\end{proof}

\begin{proposition}\label{prop aj>0}
    At least one of $w_j$, $1\le j\le p$, has a positive real part.
\end{proposition}
\begin{proof}
    Suppose on the contrary that $a_j=\Re(w_j)\le 0$ for all $1\le j\le p$. Let $\PR(z)=\sum_{j=1}^p c_jz^j$ where $c_p=1$, then $\QR(z)=\sum_{j=1}^p (-1)^{p-j}\bar c_jz^j$. Then an easy computation shows that the coefficient of $z^{2p-2}$ in \eqref{eq PR QR} is
    \begin{equation*}
        -2p+c_{p-1}+\bar c_{p-1}=0.
    \end{equation*}
    This implies
    \begin{equation*}
        \sum_{j=1}^p a_j=-\Re(c_{p-1})=-p.
    \end{equation*}

    Without loss of generality, we may assume $a_1=\max\{a_j: 1\le j\le p\}$. Then $-1\le a_j\le 0$.
    Now the real part of \eqref{eq wj} for $j=1$ gives that
    \begin{equation*}
        \sum_{k=2}^p \frac{a_1-a_k}{|w_1-w_k|^2}-\sum_{k=2}^p \frac{a_1+a_k}{|w_1+\bar w_k|^2}-\frac{1}{2a_1}-\frac{1}{2}=0.
    \end{equation*}
However, $a_1-a_k\ge 0\ge a_1+a_k$ for all $k$, and $-\frac{1}{2a_1}-\frac{1}{2}\ge 0$, so the above equality cannot hold. This is a contradiction, and the proposition is proved.
\end{proof}

Now let $a_j=\Re(w_j)>0$. Then $r_j>1$ for $\beta$ large enough, so $z_j$ is outside the unit disk for $\beta$ large enough, which is a contradiction. Finally, by the above discussion and Proposition \ref{prop Im phi_j}, there is no solution $P(z)$ if $m\ge \beta=1+\alpha$. Thus we complete the proof of Theorem \ref{thm Phim_class}.

	\begin{remark}
		It is worth emphasizing that the proof of the classification theorem uses only two features of $P(z)$: its smooth dependence on the parameter $\beta$ and the relation
		\[
		|P_{\infty}(0)|=1.
		\]
		Thus the argument is largely independent of the explicit formula for $P(z)$. It would be interesting to revisit the classification problem from a more direct algebraic perspective and derive the theorem by exploiting the explicit structure of $P(z)$.
	\end{remark}
	
	\begin{remark}
		Another natural question is to classify all polynomials $P(z)$ satisfying \eqref{eq PR QR}. For small values of $p$, this problem can be approached by writing down the relations among the coefficients and solving them explicitly. In particular, no such polynomial exists when $p=2$. For $p=3$, one finds
		\[
		\PR(z)=(z+1+ti)^3-4+si, \quad t,s\in\mathbb R,
		\]
		whereas for $p=4$ one obtains
		\[
		\PR(z)=(z+2+ti)^3(z-2+ti), \quad t\in\mathbb R.
		\]
		These examples indicate that a complete classification of the solutions to \eqref{eq PR QR} may be within reach, possibly through a more systematic development of the algebraic ideas used above.
	\end{remark}
    
\section{Local uniqueness results}\label{sec unique}

After the classification results obtained in the previous section, the remaining task is to show that each admissible blow-up pattern gives rise to a unique branch of solutions (modulo rotation). In this section we gather all local-uniqueness results needed in the proof of the main theorem, while separating the analysis according to the nature of the blow-up profile.
		
		More precisely, we distinguish three cases. The first concerns the non-singular $m$-peak solutions. The second and third cases concern singular bubbling at the origin, depending on whether the singular source  is quantized ($\alpha\in\mathbb{N}$) or non-quantized ($\alpha\notin\mathbb{N}$). The overall philosophy is to establish the symmetry of the solutions first, but each case still requires a separate treatment because the underlying blow-up mechanisms are different.
		For this reason, we divide the section into three subsections and discuss each case separately.
        
        \subsection{Local uniqueness of $m$-peak bubble solutions}
        \begin{proposition}\label{prop uniqueness-1}
        Let $m<\al+1$, and let $u_n^{(1)}$ and $u_n^{(2)}$ be two sequences of $m$-peak bubble solutions concentrating at the same blow-up points. Then there exists a large $n_0>1$ such that 
            \[u_n^{(1)}\equiv u_n^{(2)} ,\quad\text{for}~n>n_0.\]
        \end{proposition}

        To prove the local uniqueness of $m$-peak bubble solutions, we first need to understand their properties.  
		Let $u_n(x)$ be a sequence of $m$-peak bubble solutions. By Theorem \ref{thm Phim_class}, after applying a rotation, we may assume that the blow-up points of $u_n$ are 
        \begin{equation}\label{critical point}
            q_j=r_{\alpha,m}e^{i(j-1)\frac{2\pi}{m}}, \quad\forall~j=1,\cdots,m.
        \end{equation}
        As we discussed in the introduction, it is important to study the Hessian of $\Phi_m$ at $\boldsymbol{q}=(q_1,\cdots, q_m)$, which is expected to be degenerate due to the rotational symmetry. We have the following result:
		\begin{lemma}\label{lem evalue H}
			Let $H=D^2\Phi_m(\boldsymbol{q})$. Then $H$ is degenerate, and it has a unique eigenvector (up to a scalar) 
			\begin{equation*}
			\boldsymbol{v}=	(0,1,-\sin\vartheta_2, \cos\vartheta_2, \cdots, -\sin\vartheta_{m}, \cos\vartheta_{m})
			\end{equation*}
			corresponding to the zero eigenvalue, where $\vartheta_j=\f{2\pi}{m}(j-1)$, $1\le j\le m$.
		\end{lemma}
		The proof of Lemma \ref{lem evalue H} is a straightforward but lengthy computation, and we leave it to the appendix.

        \begin{lemma}\label{lemma symmetry-1}
        Let $u_n(x)$ be a sequence of $m$-peak bubble solutions. Then there exists a large $n_0>1$ such that for $n>n_0$, $u_n$ is symmetric with respect to the lines joining the origin $O$ and each local maximum point of $u_n$.
        \end{lemma}
        \begin{proof}
		Let $x_{n,j}$ be all local maximum points of $u_n$, i.e.,
            \[u_n(x_{n,j})=\max_{B_r(q_j)} u_n,\quad\forall~j=1,\cdots,m,\] 
        where $r>0$ is small. Let
            \[\mu_{n,j}=u_n(x_{n,j}),\quad \mu_n=\max_{1\le j\le m} \mu_{n,j}.\] 
        Write $\boldsymbol{x}_n=(x_{n,1},\cdots,x_{n,m})$. Then it is known that
		\begin{equation}\label{DPhi small}
			D \Phi_m(\boldsymbol{x}_n)=O(\mu_{n}e^{-\mu_{n}}).  
		\end{equation}
		By Taylor expansion and thet fact that $\boldsymbol{q}$ is critical point of $\Phi_m$, we have
        \begin{equation}\label{DPhi small-1}
            D^2\Phi_m(\boldsymbol{q})\cdot(\boldsymbol{x}_n-\boldsymbol{q})^T+O(|\boldsymbol{x}_n-\boldsymbol{q}|^2)=O(\mu_{n}e^{-\mu_{n}}).
        \end{equation}
        In general, since $H=D^2\Phi_m(\boldsymbol{q})$ is degenerate, we cannot deduce the estimate of $|\boldsymbol{x}_n-\boldsymbol{q}|$ directly.
	    To deal with this degeneracy, by taking a suitable rotation, we assume that $x_{n,1}$ lies on the positive $x_1$-axis, and hence in polar coordinates
		\begin{equation}\label{x-q perp v}
			x_{n,1}-q_1=(r_{n,1},0),\quad x_{n,j}-q_j=(r_{n,j},\theta_{n,j}),~\forall~j=2,\cdots,m.
		\end{equation}
		We claim that \eqref{DPhi small} implies $|\boldsymbol{x}_n-\boldsymbol{q}|=O(\mu_{n}e^{-\mu_{n}})$.
        By the discussion about the Hessian matrix in the Appendix, we get that
            \[D^2\Phi_m(\boldsymbol{q})=A^\star\cdot \text{Diag}(M_0,M_1,\cdots,M_{m-1})\cdot A,\]
        where $M_p$ are $2\times2$ complex matrices given by \eqref{Mp}, $A^\star$ is the conjugate transpose matrix of $A$, and $A$ is an invertible $2m\times2m$ complex matrix given by
            \[A=\f{1}{m}(F_m^\star\otimes E_2)=\f1m\left(\begin{matrix}
                A_{11} & A_{12} \\
                A_{21} & A_{22}
            \end{matrix}\right)\]
        with 
            \[A_{11}=\left(\begin{matrix}
                1 & 0\\
                0 & 1
            \end{matrix}\right), \quad A_{12}=A_{21}^T=\left(\begin{matrix}
                1 & 0 & \cdots & 1 & 0 \\
                0 & 1 & \cdots & 0 & 1
            \end{matrix}\right)_{2\times 2(m-1)},\]
        and $A_{22}$ is a $2(m-1)\times2(m-1)$ complex matrix. Moreover, we have
            \[A^{-1}=mA^\star=\left(\begin{matrix}
                A_{11} & A_{12} \\
                A_{21} & A_{22}^\star
            \end{matrix}\right).\]
        Let 
            \[\boldsymbol{s}_n=(s_{n,1},s_{n,2},\cdots,s_{n,2m})^T=A\cdot(\boldsymbol{x}_n-\boldsymbol{q})^T,\]
        and
            \[\boldsymbol{\ti s}_n=(s_{n,1},s_{n,3},\cdots,s_{n,2m})^T.\]
        Then by \eqref{x-q perp v}, we have
            \[0=(A^{-1}\cdot\boldsymbol{s}_n)_2=s_{n,2}+\sum_{j=2}^ms_{n,2j},\]
        so that
            \[|\boldsymbol{\ti s}_n|\simeq|\boldsymbol{s}_n|\simeq|\boldsymbol{x}_n-\boldsymbol{q}|.\]
        Thus by \eqref{DPhi small-1} and \eqref{M0}, we get that
             \[\text{Diag}\bigg(-\frac{2(\beta^2-m^2)}{\rho},0,M_1,\cdots,M_{m-1}\bigg)\cdot \boldsymbol{\ti s}_n=O(|\boldsymbol{\ti s}_n|^2)+O(\mu_{n}e^{-\mu_{n}}).\]
        which together with  \eqref{det Mp} gives $|\boldsymbol{\ti s}_n|=O(\mu_{n}e^{-\mu_{n}})$.
        Thus we have proved $|\boldsymbol{x}_n-\boldsymbol{q}|=O(\mu_{n}e^{-\mu_{n}})$.
        
		Now we study the symmetric property of $u_n$. 
        Since $x_{n,1}$ lies on the positive $x_1$-axis by assumption, we are going to show that $u_n$ is symmetric with respect to $x_1$-axis. 
        Let $u_{n,1}$ be the reflection of $u_n$  with respect to $x_1$-axis i.e., $u_{n,1}(x_1,x_2)=u_n(x_1,-x_2)$. 
        Then $u_{n,1}$ is also a $m$-peak bubble solution concentrating at $q_1,\cdots,q_m$. 
        To prove $u_n\equiv u_{n,1}$ for large $n$, we basically follow the strategy in \cite{Bartolucci-Jevnikar-Lee-Yang19-2}, where, under a non-degeneracy assumption, local uniqueness of multi-peak bubble solutions has been studied. 
        Here, although we have no non-degenerate condition, we also get the estimates of $|\boldsymbol{x}_n-\boldsymbol{q}|$, so that Theorem 2A, 2B in \cite{Bartolucci-Jevnikar-Lee-Yang19-2} still hold, and then we can follow step by step their approach. 
        For simplicity, we omit most details and only give some necessary computations without non-degenerate condition.
        Suppose by contradiction that $u_n\not\equiv u_{n,1}$, and let
		\begin{equation*}
			\xi_n(x)=\frac{u_n(x)-u_{n,1}(x)}{\|u_n-u_{n,1}\|_{\infty}}.
		\end{equation*}
		For each $j=1,\cdots,m$, we define the scaling functions $\xi_{n,j}(z):=\xi_n(e^{-\frac{\mu_{n,j}-\log 8}{2}}z+x_{n,j})$, then by Lemma 3.2 in \cite{Bartolucci-Jevnikar-Lee-Yang19-2}, there exist $b_{j,0}, b_{j,1}$ and $b_{j,2}$ such that
		\begin{equation*}
			\xi_{n,j}(z)\to \sum_{l=0}^2b_{j,l}\psi_l, \text{ in } C^2_{\loc}(\R^2),
		\end{equation*}
		where
		\begin{equation*}
			\psi_0(z)=\frac{1-|z|^2}{1+|z|^2},\quad \psi_1(z)=\frac{z_1}{1+|z|^2},\quad
			\psi_2(z)=\frac{z_2}{1+|z|^2}.
		\end{equation*}
		By Lemma 3.3 in \cite{Bartolucci-Jevnikar-Lee-Yang19-2}, we have 
            \[b_{j,0}=0,~\forall~j, \quad\text{and}\quad  D^2\Phi_m(\boldsymbol{q})\cdot\boldsymbol{b}=\boldsymbol{0},\]
        where $\boldsymbol{b}=(b_{1,1},b_{1,2},\cdots, b_{m,1},b_{m,2})$. 
        By Lemma \ref{lem evalue H}, $\boldsymbol{b}=c\boldsymbol{v}$ for some number $c$. On the other hand, since $x_{n,1}$ is a local maximum for both $u_n(x)$ and $u_{n,1}(x)$, we have $\nabla\xi_{n,1}(0)=0$, which implies $b_{1,1}=b_{1,2}=0$. Therefore $\boldsymbol{b}=0$, and the same argument in \cite[Page 457]{Bartolucci-Jevnikar-Lee-Yang19-2} shows that $\xi_n\equiv 0$, which is a contradiction to $\|\xi_n\|_\infty=1$.
        Thus we have finished the proof that $u_n\equiv u_{n,1}$ for large $n$.
		\end{proof}

        \begin{proof}[Proof of Proposition \ref{prop uniqueness-1}]
		Applying Lemma \ref{lemma symmetry-1}, we find that $x_{n,j}$ must be the vertices of a regular $m$-polygon. Moreover, the only configuration of $\boldsymbol{x}_n$ such that \eqref{x-q perp v} holds is the case that $O, q_j, x_{n,j}$ are on the same line for all $j$. Now let $u_n^{(1)}$ and $u_n^{(2)}$ be two $m$-peak solutions such that their local maximum points satisfy this configuration. By symmetry,
		\begin{equation*}
			\partial_{x_2}(u_n^{(1)}-u_n^{(2)})\Big|_{x=x_{n,1}^{(1)}}=0.
		\end{equation*}
		Then the same argument as in the proof of Lemma \ref{lemma symmetry-1} shows that $u_n^{(1)}\equiv u_n^{(2)}$ for large $n$, and the local uniqueness is proved.
        \end{proof}

        \vskip0.1in
        \subsection{Local uniqueness of non-quantized singular bubble solution}
        \begin{proposition}\label{prop uniqueness-2}
            Let $\al\not\in\N$, and let $u_n$ be a sequence of non-quantized singular bubble solutions of \eqref{equ-2}. Then there exists a large $n_0>1$ such that 
              \begin{equation}\label{un sol non-quantized}
    u_n(x)=\log\f{\Lambda_{n,2}}{\sbr{1+\f{\Lambda_{n,2}}{8(1+\al)^2}|x|^{2(1+\al)}}^2}-\log\la_n,\quad\text{for}~n>n_0,        
        \end{equation}
            where $\Lambda_{n,2}$ is the larger solution of
                \[\f{\Lambda}{\sbr{1+\f{\Lambda}{8(1+\al)^2}}^2}=\la_n.\]
        \end{proposition}

        In \cite{Bartolucci-Jevnikar-Lee-Yang20}, under the assumption that the singular point $0$ is a non-degenerate critical point of some Kirchhoff-Routh type function, the local uniqueness was established. However, it is easy to check that for the unit disk, this non-degeneracy no longer holds, and hence we do not adopt the traditional approach in \cite{Bartolucci-Jevnikar-Lee-Yang20}. 
        %In \cite{Li-Li-Wei-25}, the authors have gave some refined estimates for singular bubbling solution, and we believe that these refined estimates can be used to remove the assumption in \cite{Bartolucci-Jevnikar-Lee-Yang20}, but here we give a simpler proof in the unit disk.
        Thanks to the symmetry of the disk, we prove the axial symmetry property of non-quantized singular bubble solutions, and then study the radial symmetry to get local uniqueness.
        
        We first list some properties. The following results are standard and can be found, for example, in \cite[Section 2]{Bartolucci-Jevnikar-Lee-Yang20}.
        Let $u_n$ be a sequence of non-quantized singular bubble solutions of \eqref{equ-2}. 
        Denote
            \[\rho_n=\la_n\int_{B}|x|^{2\al}e^{u_n}\rd x.\]
        Then $\rho_n\to8\pi(1+\al)$ and $u_n$ satisfies the following mean-field equation
            \[\Delta u_n+\rho_n\f{|x|^{2\al}e^{u_n}}{\int_{B}|x|^{2\al}e^{u_n}\rd x}=0.\]
        By (2.5)-(2.13) of \cite{Bartolucci-Jevnikar-Lee-Yang20}, we immediately get 
        \begin{equation}\label{in un-1}
            u_n(x)+\log\la_n=\log\f{e^{\mu_n}}{\sbr{1+\f{e^{\mu_n}}{8(1+\al)^2}|x|^{2(1+\al)}}^2}+o(1),\quad\forall~x\in B_{r_0}(0),
        \end{equation}
        where $r_0>0$ is a small fixed constant, and 
            \[\mu_n=\max_B u_n+\log\la_n\to+\iy.\]
        Moreover, we have
            \[\mu_n+\log\la_n=O(1),\]
        and
        \begin{equation}\label{out un-1}
            u_n(x)=-4(1+\al)\log|x|+o(1),\quad\text{in}~C_\loc^1(\ov B\setminus \{0\}).
        \end{equation}

        \begin{lemma}\label{lemma symmetry-2}
        Let $u_n$ be a sequence of non-quantized singular bubble solutions of \eqref{equ-2}. Then there exists a large $n_0>1$ such that $u_n$ is symmetric with respect to $x_1$-axis for $n>n_0$. 
        \end{lemma}
        \begin{proof}
        Let $u_{n,1}$ be the reflection of $u_n$ with respect to $x_1$-axis i.e. $u_{n,1}(x_1,x_2)=u_n(x_1,-x_2)$.
        Suppose by contradiction that $u_n\not\equiv u_{n,1}$ for each $n$. Let
            \[w_n=u_n-u_{n,1}\not\equiv0, \quad\text{in}~B.\]
        It is easy to check that
            \[\Delta w_n+|x|^{2\al}c_nw_n=0~~\text{in}~B,\quad w_n=0~~\text{on}~\pa B,\]
        where  
            \[c_n=\la_n\f{e^{u_n}-e^{u_{n,1}}}{u_n-u_{n,1}}.\]
        Set 
            \[N_n=|w_n(x_n)|=\max_B |w_n|>0.\]
        We claim that
        \begin{equation}\label{3-20-1}
            |x_n|\to0\quad\text{and} \quad |x_n|e^{\f{\mu_n}{2(1+\al)}}\to+\iy.
        \end{equation}
        By \eqref{out un-1}, we have $c_n=O(\la_n)$ uniformly in $C_\loc^1(\ov B\setminus \{0\})$, and hence we get that $\f{w_n}{N_n}\to w$ in $C_\loc(\ov B\setminus \{0\})$ with
            \[\Delta w=0~~\text{in}~B,\quad |w|\le 1~~\text{in}~B, \quad w=0~~\text{on}~\pa B.\]
        Thus $w=0$ and $x_n\to0$. Let $\ep_n=e^{-\f{\mu_n}{2(1+\al)}}\to0$ and $\ti w_n(x)=w_n(\ep_nx)$ for $x\in B_{\f{1}{\ep_n}}(0)$.
        Then we have
            \[\Delta\ti w_n+|x|^{2\al}\ti c_n\ti w_n=0,\quad\text{in}~B_{\f{1}{\ep_n}}(0),\]
        where by \eqref{in un-1} we have
            \[\ti c_n(x)=\ep_n^{2(1+\al)}c_n(\ep_nx)=\f{1+o(1)}{\sbr{1+\f{1}{8(1+\al)^2}|x|^{2(1+\al)}}^2},\quad\text{in}~B_{\f{r_0}{\ep_n}}(0).\]
        Thus we get $\f{\ti w_n}{N_n}\to\ti w=a\f{8(1+\al)^2-|x|^{2(1+\al)}}{8(1+\al)^2+|x|^{2(1+\al)}}$. 
        Since $\ti w_n(0)=0$, we get $\ti w=0$ and hence prove \eqref{3-20-1}.

        Now we use \eqref{3-20-1} to get a contradiction. 
        By Green's representation formula, we have that for any $y\in B_{\f{1}{\ep_n}}(0)$,
            \[\ti w_n(y)=\int_{B_{\f{1}{\ep_n}}(0)}G_n(y,x)|x|^{2\al}\ti b_n(x)\ti w_n(x)\rd x,\]
        where $G_n(y,x)$ is the Green's function 
            \[G_n(y,x)=-\f{1}{2\pi}\log|\ep_n(y-x)|+\f{1}{2\pi}\log\abs{\ep_n^2|y|x-\f{y}{|y|}}.\]
        Set $y_n=\f{x_n}{\ep_n}$ and $\bar y_n=(y_{n,1},-y_{n,2})$, then \eqref{3-20-1} implies $|y_n|\to+\iy$ and $\ep_n|y_n|=o(1)$.
        Since 
            \[G_n(\bar y_n,x)-G_n(y_n,x)=\f{1}{2\pi}\log\f{|y_n-x|}{|\bar y_n-x|}+o(1),\quad\text{uniformly for}~x\in B_{\f{1}{\ep_n}}(0),\]
        and by \eqref{in un-1}-\eqref{out un-1}
            \[\abs{\ti c(x)}\le\f{C}{1+|x|^{4(1+\al)}}\quad\text{uniformly for}~x\in B_{\f{1}{\ep_n}}(0),\]
        we get that
        \begin{equation}\label{3-20-6}
            \abs{\ti w_n(\bar y_n)-\ti w_n(y_n)}\le CN_n\int_{B_{\f{1}{\ep_n}}(0)}\abs{\log\f{|y_n-x|}{|\bar y_n-x|}}\f{|x|^{2\al}}{1+|x|^{4(1+\al)}}\rd x+o(N_n).
        \end{equation}
        We show that
        \begin{equation}\label{3-20-2}
            \int_{B_{\f{1}{\ep_n}}(0)}\abs{\log\f{|y_n-x|}{|\bar y_n-x|}}\f{|x|^{2\al}}{1+|x|^{4(1+\al)}}\rd x=o(1).
        \end{equation}
        Let 
            \[\Omega_n=B_{\f{1}{\ep_n}}(0)\setminus\sbr{B_{\f12|y_n|}(y_n)\cup B_{\f12|y_n|}(\bar y_n)}.\]
        For some fixed large $R>1$, since for any $x\in\Omega_n\cap B_{R|y_n|}(0)$,
            \[\f12\le\abs{\f{y_n}{|y_n|}-\f{z_n}{|y_n|}}\le R+1, \]
        we get 
            \[\abs{\log\f{|y_n-x|}{|y_n|}}=O(\f{|x|}{|y_n|}),\quad\text{for}~x\in\Omega_n\cap B_{R|y_n|}(0),\]
        which gives 
        \begin{equation}\label{3-20-3}
             \abs{\log\f{|y_n-x|}{|\bar y_n-x|}}=O\sbr{\f{|x|}{|y_n|}}\quad\text{for}~x\in\Omega_n\cap B_{R|y_n|}(0).
        \end{equation}
        On the other hand, for any $x\in\Omega_n\setminus B_{R|y_n|}(0)$,
            \[\abs{\f{|y_n-x|}{|\bar y_n-x|}-1}=\abs{\f{4x_2y_{n,2}}{|x-\bar y_n|(|x-y_n|+|x-\bar y_n|)}}\le\f{2|y_n|}{(1-\f1R)^2|x|}\le\f12,\]
        we get 
        \begin{equation}\label{3-20-4}
            \abs{\log\f{|y_n-x|}{|\bar y_n-x|}}=O\sbr{\f{|y_n|}{|x|}}\quad\text{for}~x\in\Omega_n\setminus B_{R|y_n|}(0).
        \end{equation}
        By \eqref{3-20-3} and \eqref{3-20-4}, we immediately get that
        \begin{equation}\label{3-20-5}
            \int_{\Omega_n}\abs{\log\f{|y_n-x|}{|\bar y_n-x|}}\f{|x|^{2\al}}{1+|x|^{4(1+\al)}}\rd x=O\sbr{\f{1}{|y_n|}}=o(1).
        \end{equation}
        Finally, we have
            \[\begin{aligned}   
            \int_{B_{\f12|y_n|}(y_n)}\abs{\log\f{|y_n-x|}{|\bar y_n-x|}}\f{|x|^{2\al}}{1+|x|^{4(1+\al)}}\rd x &\le\f{C}{|y_n|^{4+2\al}}\int_{B_{3|y_n|}(0)}|\log|x||\rd x\\
            &=O\sbr{\f{\log|y_n|}{|y_n|^{2(1+\al)}}}=o(1),                 
            \end{aligned}\]
        and similarly the integral in $B_{\f12|y_n|}(\bar y_n)$ is also $o(1)$. 
        This together with \eqref{3-20-5} gives \eqref{3-20-2}.
        By \eqref{3-20-6} and \eqref{3-20-2} we get
            \[\ti w_n(\bar y_n)=\ti w_n(y_n)+o(N_n).\]
        However, by the definition of $w_n$ we have
            \[0=\ti w_n(y_n)+\ti w_n(\bar y_n)=\pm2N_n+o(N_n),\]
        which is in contradiction with $N_n>0$. Hence the lemma follows.
        \end{proof}
        
    \begin{proof}[Proof of Proposition \ref{prop uniqueness-2}]
    We prove that $u_n$ is radial symmetric for large $n$. 
    Let 
        \[w_n(x)=\pa_\theta u_n=-x_2\pa_{x_1}u_n+x_1\pa_{x_2}u_n,\quad\text{in}~B.\]
    Then since $u_n$ is smooth, we get that
        \[\Delta w_n+\lambda_n|x|^{2\al}e^{u_n}w_n=0~~\text{in}~B,\quad w_n=0~~\quad\text{on}~\pa B.\]
    Moreover, by Lemma \ref{lemma symmetry-2} $u_n(x_1,x_2)=u_n(x_1,-x_2)$ for large $n$, we get that
        \[w_n(x_1,-x_2)+w_n(x_1,x_2)=0.\]
    Therefore, by repeating the approach in the proof of Lemma \ref{lemma symmetry-2}, we get that $w_n\equiv0$ for large $n$, which means that $u_n$ is radial symmetric. By ODE methods, we get the desired uniqueness.
    \end{proof}

    \vskip0.1in
    \subsection{Local uniqueness of quantized singular bubble solutions}
    \begin{proposition}\label{prop uniqueness-3}
        Let $\al\in\N$, and let $u_n$ be a sequence of quantized singular bubble solutions of \eqref{equ-2}. Then there exists a large $n_0>1$ such that 
        \begin{equation}\label{un sol quantized}
    u_n(x)=\log\f{\Lambda_{n,2}}{\sbr{1+\f{\Lambda_{n,2}}{8(1+\al)^2}|x|^{2(1+\al)}}^2}-\log\la_n,\quad\text{for}~n>n_0,        
        \end{equation}
            where $\Lambda_{n,2}$ is the larger solution of
            \[\f{\Lambda}{\sbr{1+\f{\Lambda}{8(1+\al)^2}}^2}=\la_n.\]
    \end{proposition}

   Compared with the non-quantized case, the quantized case is more delicate.
    Let $u_n$ be a sequence of quantized singular bubble solutions of \eqref{equ-2}. 
    By \cite{Teresa-Wei-Zhang25}, we know that $u_n$ must be simple blow-up. Then it is well known \cite{Kuo-Lin,Bartolucci-Tarantello} that a suitable scaling of $u_n$ converges to the function
        \[U_{\Lambda,\xi}(x)=\log\f{\Lambda^2}{\sbr{1+\f{\Lambda^2}{8(1+\al)^2}|x^{1+\al}-\xi|^2}^2},\quad\Lambda>0,\xi\in\C,\]
    which are all solutions of the limit problem
    \begin{equation}\label{limit problem} 
        \Delta U+|x|^{2\al}e^U=0,\quad\text{in}~\R^2,\quad\int_{\R^2}|x|^{2\al}e^U\rd x<+\infty
    \end{equation}
     by the classification result of Prajapat and Tarantello  \cite{Prajapat-Tarantello}.  
    The difference between the non-quantized case and the quantized case is that for $\al\in\N$, the bounded kernel of linear equation of \eqref{limit problem} is
        \[\text{span}\lbr{\pa_\Lambda U_{\Lambda,\xi},\pa_\xi U_{\Lambda,\xi},\pa_{\bar\xi} U_{\Lambda,\xi}},\]
        %\[\text{span}\lbr{\f{8(1+\al)^2-\Lambda^2|x^{1+\al}-\xi|^2}{8(1+\al)^2+\Lambda^2|x^{1+\al}-\xi|^2},\f{x^{1+\al}-\xi}{8(1+\al)^2+\Lambda^2|x^{1+\al}-\xi|^2}, \f{\ov{x^{1+\al}-\xi}}{8(1+\al)^2+\Lambda^2|x^{1+\al}-\xi|^2}}\]
    where $\bar \xi$ is the conjugate of $\xi$ (see \cite{delPino-1}). 
    To get local uniqueness of quantized singular bubble solutions, we need to locate $\xi$.
    Fortunately, applying \cite[Theorem 3]{Li-Li-Wei-25} in the unit disk, one must have $\xi=0$, and the following expansion of $u_n$ holds:
        \begin{equation}\label{in un-2}
            u_n(x)+\log\la_n=\log\f{e^{\mu_n}}{\sbr{1+\f{e^{\mu_n}}{8(1+\al)^2}|x|^{2(1+\al)}}^2}+o(1),\quad\forall~x\in B_{r_0}(0),
        \end{equation}
        where $r_0>0$ is a small fixed constant, and 
            \[\mu_n=\max_B u_n+\log\la_n\to+\iy.\]
        Moreover, by \cite[Lemma 4.1]{Li-Li-Wei-25}, we have
        \begin{equation}\label{out un-2}
            u_n(x)=-4(1+\al)\log|x|+o(1),\quad\text{in}~C_\loc^1(\ov B\setminus \{0\}).
        \end{equation}
        Then comparing \eqref{in un-2} and \eqref{out un-2}, we get 
            \[\mu_n+\log\la_n=O(1).\]
    First we prove that
    
    \begin{lemma}\label{lemma symmetry-3}
    Let $u_n$ be a sequence of quantized singular bubble solutions of \eqref{equ-2}. Then there exists a large $n_0>1$ such that $u_n$ is $\f{2\pi}{1+\al}$-rotationally symmetric for all $n>n_0$, i.e.,
        \[u_n(x)=u_n(xe^{i\f{2\pi}{1+\al}}),\quad\forall~x\in B,\quad\text{for}~n>n_0.\]
    \end{lemma}
    \begin{proof}
    The proof is similar to that of Lemma \ref{lemma symmetry-1}. 
    Let $u_{n,1}(x)=u_n(xe^{i\f{2\pi}{1+\al}})$. Suppose by contradiction that $u_n\not\equiv u_{n,1}$ for each $n$. Let
            \[w_n=u_n-u_{n,1}\not\equiv0, \quad\text{in}~B.\]
        It is easy to check that
            \[\Delta w_n+|x|^{2\al}c_nw_n=0~~\text{in}~B,\quad w_n=0~~\text{on}~\pa B,\]
        where  
            \[c_n=\la_n\f{e^{u_n}-e^{u_{n,1}}}{u_n-u_{n,1}}.\]
        Set 
            \[N_n=|w_n(x_n)|=\max_B |w_n|>0.\]
        We claim that
        \begin{equation}\label{3-20-10} 
            |x_n|\to0\quad\text{and} \quad |x_n|e^{\f{\mu_n}{2(1+\al)}}\to+\iy.
        \end{equation}
        By \eqref{out un-2}, we have $c_n=O(\la_n)$ uniformly in $C_\loc^1(\ov B\setminus \{0\})$, and hence we get that $\f{w_n}{N_n}\to 0$ in $C_\loc(\ov B\setminus \{0\})$, so that $x_n\to0$. 
        Let $\ep_n=e^{-\f{\mu_n}{2(1+\al)}}\to0$ and $\ti w_n(x)=w_n(\ep_nx)$ for $x\in B_{\f{1}{\ep_n}}(0)$.
        Then we have
            \[\Delta\ti w_n+|x|^{2\al}\ti c_n\ti w_n=0,\quad\text{in}~B_{\f{1}{\ep_n}}(0),\]
        where by \eqref{in un-2} we have
            \[\ti c_n(x)=\ep_n^{2(1+\al)}c_n(\ep_nx)=\f{1+o(1)}{\sbr{1+\f{1}{8(1+\al)^2}|x|^{2(1+\al)}}^2},\quad\text{in}~B_{\f{r_0}{\ep_n}}(0).\]
        Thus we get 
            \[\f{\ti w_n}{N_n}\to\ti w=a_0\f{8(1+\al)^2-|x|^{2(1+\al)}}{8(1+\al)^2+|x|^{2(1+\al)}}+a_1\f{x^{1+\al}}{8(1+\al)^2+|x|^{2(1+\al)}}+a_2\f{\bar x^{1+\al}}{8(1+\al)^2+|x|^{2(1+\al)}}.\] 
        in $C_\loc(\R^2)$ where $a_0,a_1,a_2$ are complex constants. 
        Since $\ti w_n(0)=0$, we get $a_0=0$. By the definition of $w_n$, we get $\sum_{j=0}^{\al} w_n(xe^{i\f{2\pi}{1+\al}j})=0$, so that
            \[ 0=\sum_{j=0}^{\al}\ti w(xe^{i\f{2\pi}{1+\al}j})=(1+\al)\f{a_1x^{1+\al}+a_2\bar x^{1+\al}}{8(1+\al)^2+|x|^{2(1+\al)}}.\]
        This implies $a_1=a_2=0$, and hence proves \eqref{3-20-10}.
        Then by using \eqref{3-20-10} and \eqref{in un-2}-\eqref{out un-2}, and following the discussion between \eqref{3-20-6} and \eqref{3-20-5}, we get 
            \[\ti w_n(y_ne^{i\f{2\pi}{1+\al}j})=\ti w_n(y_n)+o(N_n), \quad\forall~j=1,\cdots,\al,\]
        where $y_n=\f{x_n}{\ep_n}$. 
        Then we deduce
            \[0=\sum_{j=0}^\al\ti w_n(y_ne^{i\f{2\pi}{1+\al}j})=(1+\al)\ti w_n(y_n)+o(N_n)=\pm(1+\al)N_n+o(N_n),\]
        which is in contradiction with $N_n>0$. This proves the lemma.
    \end{proof}

    \begin{proof}[Proof of Proposition \ref{prop uniqueness-3}]
    We prove that $u_n$ is radial symmetric for large $n$. By Lemma \ref{lemma symmetry-3}, $u_n(x)=u_n(xe^{i\f{2\pi}{1+\al}})$ for large $n$, so that $v_n(x)=u_n(x^{\f{1}{1+\al}})$ is well-defined for $x\in B$. Moreover, by direct computation, we have
        \[\Delta v_n+\f{\la_n}{(1+\al)^2}e^{v_n}=0~~\text{in}~B,\quad v_n=0~~\text{on}~\pa B.\]
    Thus by the result in \cite{Gidas-Ni-Nirenberg79}, we get that $u_n$ is radial symmetric for $n$ large. By ODE methods, we get the desired uniqueness.
    \end{proof}

\section{Proof of Theorem \ref{thm main_class}}\label{proof1.1_class}
In this section, we complete the proof of Theorem \ref{thm main_class} and record several additional properties of the resulting solutions.

The singular bubble solutions are already given explicitly by \eqref{un sol non-quantized} and \eqref{un sol quantized}. It remains to establish the existence of the $m$-peak bubble solutions.

For $m=1$, as pointed out in \cite{Bartolucci-Montefusco06}, the existence of $1$-peak solutions is given by Theorem 3 in \cite{del Pino-Kowalczyk-Musso05} or Theorem 1.1 in \cite{Esposito-Grossi-Pistoia05}. Let $u_n$ be a $1$-peak solution to \eqref{equ-2}, then it is straightforward to verify that $v_n(x)=u_n(x^m)$ satisfies the equation
		\begin{equation*}
			\Delta v_n+m^2{\lambda_n{\tiny }} |x|^{2(m(\alpha+1)-1)}e^{v_n}=0.
		\end{equation*}
        Moreover, $v_n$ is $\frac{2\pi}{m}$-rotationally symmetric, and the maximum point of $u_n$ splits into $m$ local maxima of $v_n$, so $v_n$ is an $m$-peak solution. 
        
       To summarize, we have shown that we can generate $m$-peak solutions from the $1$-peak solutions for any positive integer $m$, and the local uniqueness implies that all $m$-peak solutions must be in this form. This completes the classification and the proof of Theorem \ref{thm main_class}.

From the relation between $m$-peak solutions and $1$-peak solutions, we can immediately deduce the following corollary of Theorem 3.1(d) and Theorem 3.3 in \cite{Bartolucci-Montefusco06}:
\begin{corollary}\label{coro m peak}
    Let $u_n$ be a sequence of $m$-peak solutions to \eqref{equ-2}. After a suitable rotation, we can assume that $u_n$ has a local maximum point on the positive $x_1$-axis. Then
    \begin{enumerate}[label=(\roman*)]
        \item The mass satisfies $\int_B \lambda_n|x|^{2\alpha}e^{u_n}\rd x>8m\pi$.
    \item  For any fixed $r_0\in (0,1)$, the function $u_n(r_0\cos\vartheta, r_0\sin\vartheta)$ is strictly decreasing for $\vartheta\in (\frac{2k\pi}{m}, \frac{(2k+1)\pi}{m})$, and  strictly increasing for $\vartheta\in (\frac{(2k+1)\pi}{m}, \frac{(2k+2)\pi}{m})$ for $k=0, \cdots, m-1$.    
    \end{enumerate}
\end{corollary}

\begin{remark}
    In contrast to Corollary \ref{coro m peak} $(i)$, the mass of the singular blow-up solutions is always less than its limit $8\pi(1+\alpha)$ (see \eqref{bound}). Using \eqref{un sol non-quantized} or \eqref{un sol quantized}, we can compute the mass explicitly:
    \begin{equation*}
        \begin{aligned}
            	\int_{B} \lambda_n |x|^{2\alpha}e^{u_n(x)}\rd x=&\int_{B}|x|^{2\alpha} \f{\Lambda_{n,2}}{\sbr{1+\f{\Lambda_{n,2}}{8(1+\al)^2}|x|^{2(1+\al)}}^2}\rd x\\
		  		=&\frac{\pi}{1+\alpha}\int_{0}^1  \f{\Lambda_{n,2}}{\sbr{1+\f{\Lambda_{n,2}}{8(1+\al)^2}r}^2}\rd r\\
		  				=&8(1+\alpha)\left(1-\frac{1}{1+8\Lambda_{n,2} (\alpha+1)^2}\right).
        \end{aligned}
    \end{equation*}
\end{remark}

\section{Non-degeneracy results}\label{proof1.3_class}
In this section we prove Theorem \ref{thm nondegen_class} by studying bounded solutions of the linearized equation \eqref{eq phi}. Note that by a contradiction argument, we only need to prove Theorem \ref{thm nondegen_class}  for solution sequence $u_n$ with $\la_n\to0$.
The argument depends on the type of solutions. For the $m$-peak bubble solutions, rotational invariance gives rise to a natural solution, and the goal is to show that no other bounded element lies in the kernel. 
For radial solutions, by contrast, the explicit expressions enable us to study \eqref{eq phi} by Fourier analysis. 
We therefore treat the $m$-peak case and the radial case separately in the two subsections below.

\vskip0.1in
\subsection{Non-degeneracy of $m$-peak bubble solutions} \ \\ \indent 
Let $u_n$ be a sequence of $m$-peak bubble solutions. By Theorem \ref{thm Phim_class}, after taking a rotation we assume that the blow-up points $\boldsymbol{q}=(q_1,\cdots,q_m)$ of $u_n$ satisfy \eqref{critical point}.
It is easy to see that $\pa_\theta u_n$ is a non-trivial solution of \eqref{eq phi}, and our main task is to prove that $\tau=1$ is a simple eigenvalue of 
    \[\Delta \phi+\tau\lambda_n|x|^{2\alpha}e^{ u_n}\phi=0, \quad\text{in}~B,\quad \phi=0, \text{ on } \partial B. \]
For this purpose, we need to study the expansion of eigenvalues $\tau_{n,l}$ of the following linear problems
\begin{equation}\label{linear problem}\begin{cases}
    \Delta \phi_{n,l}+\tau_{n,l}\lambda_n|x|^{2\alpha}e^{ u_n}\phi_{n,l}=0, \quad\text{in}~B,\\
    \phi_{n,l}=0, \text{ on } \partial B.
\end{cases}\end{equation}
It is well known that $0<\tau_{n,1}<\tau_{n,2}\le \tau_{n,3}\le\cdots\to+\iy$.
Linear problems of the form \eqref{linear problem} with various potentials have been widely studied in the literature, see for example \cite{Gladiali-Grossi-Ohtsuka16,Gladiali-Grossi-Ohtsuka-Suzuki14,OSS13,Ohtsuka-Sato24}.
Here we cite \cite[Theorem 1.3]{Ohtsuka-Sato24} to give that
\begin{equation}\begin{aligned}
    \tau_{n,l}&=-\f{1}{2\log\la_n}+o\sbr{\f{1}{\log\la_n}}, &\text{for}~1\le l\le m,\\
    \tau_{n,l}&=1-48\pi\eta_{2m-(l-m)+1}\la_n+o(\la_n), &\text{for}~m+1\le l\le 3m,\\
    \tau_{n,l}&>1, &\text{for}~l\ge 3m+1,\\
\end{aligned}\end{equation}
where $\eta_1\le \eta_2\le\cdots\le \eta_{2m}$ are eigenvalues of $D_0\cdot D^2\Phi_m(\boldsymbol{q})\cdot D_0$ with $D^2\Phi_m(\boldsymbol{q})$ is the Hessian matrix of $\Phi_m$ at $\boldsymbol{q}$, and $D_0=\text{Diag}(d_1,d_1,d_2,d_2,\cdots,d_m,d_m)$ for some constants $d_j>0$, $j=1,\cdots,m$.
By Lemma \ref{lem evalue H}, we get that $0$ is the simple eigenvalue of $D_0\cdot D^2\Phi_m(\boldsymbol{q})\cdot D_0$, and hence $\tau_{n,l}=1$ is a simple eigenvalue of \eqref{linear problem}. This completes the proof.

\vskip0.1in
\subsection{Non-degeneracy of radial solutions} \ \\ \indent 
The proof of this case is very close to that of Lemma 2.1 in \cite{Zhang09}. We sketch it here for the sake of completeness.

Let $u_n$ be a sequence of radial solutions with $\al>0$. Then we have
    \[u_{n}(x)=\log\f{\Lambda_{n}}{\sbr{1+\f{\Lambda_{n}}{8(1+\al)^2}|x|^{2(1+\al)}}^2}-\log\la_n,  \]
for some $\Lambda_n\to0$ or $\Lambda_n\to+\iy$. 
Let $\phi$ solve \eqref{eq phi}. For any integer $k\ge 0$, we define
			\begin{equation*}
				\phi_k(r)=\frac{1}{2\pi}\int_0^{2\pi}\phi(r\cos\theta,r\sin\theta)e^{ik\theta}\rd \theta.
			\end{equation*}
			Since $u_n$ is radially symmetric, $\phi_k$ satisfies
				\begin{equation}\label{eq phi k}
				\begin{cases}
					  \phi_k''(r)+\frac{1}{r}\phi'_k(r)+(\lambda_nr^{2\alpha}e^{u_n}-\frac{k^2}{r^2})\phi_k=0,\ 0<r<1,\\
					\phi_k(1)=0.
				\end{cases}
			\end{equation}
			Let $f_k(s)=\phi_k((\frac{\sqrt{\Lambda_n}}{8(1+\alpha)}s)^{\frac{1}{1+\alpha}})$. Then we have
			\begin{equation}\label{eq fk}
				 f_k''(s)+\frac{1}{s}f_k'(s)+\left(\f{8}{\sbr{1+s^2}^2}-\frac{k^2}{(1+\alpha)^2s^2}\right)f_k(s)=0,
			\end{equation}
            and $f_k(\f{8(1+\al)}{\sqrt\Lambda_n})=0$ for each $n$.
			Write $\delta=\frac{k}{1+\alpha}$. If $\delta\neq 1$, then the fundamental solutions of \eqref{eq fk} are
			\begin{equation*}
				f_{k1}(s)=\frac{(\delta+1)s^\delta+(\delta-1)s^{\delta+2}}{1+s^2},\quad
				f_{k2}(s)=\frac{(\delta+1)s^{2-\delta}+(\delta-1)s^{-\delta}}{1+s^2},
			\end{equation*}
			Consequently $f_k=c_1f_{k1}+c_2f_{k2}$ for some constants $c_1$ and $c_2$.  Since $f_{k2}$ is singular at $s=0$ and
				$f_{k1}(\f{8(1+\al)}{\sqrt\Lambda_n})\neq 0$ for $n$ large, we must have $c_1=c_2=0$. Hence $\phi_k\equiv 0$.

			If $\delta=1$, then this happens only when $\alpha$ is an integer. In this case
			\begin{equation*}
					f_{k1}(s)=\frac{s}{1+s^2},
			\end{equation*}
			and elementary computation shows that the other linearly independent solution of \eqref{eq fk} $f_{k2}$ satisfies $f_{k2}(s)\sim s^{-1}$ as $s\to 0$. Then for the same reason we conclude that $c_1=c_2=0$ and $\phi_k\equiv 0$. This completes the proof.
\appendix
\section{Analysis of the Matrix}\label{matrix H}

	In this appendix, we study the Hessian $H=D^2\Phi_m(\boldsymbol{q})$ where $\boldsymbol{q}=(q_1,\cdots, q_m)$. Due to the rotational symmetry, it is more convenient to study the Hessian in polar coordinates. Let $H^*$ be the Hessian in $(r_1,\theta_1,\cdots,r_{m},\theta_m)$ at $\boldsymbol{q}$. Then $H^*$ is most conveniently written as an $m\times m $ block matrix with $2\times2$ blocks
		\begin{equation*}
			H_{jk}:=\begin{pmatrix}
				\partial^2_{r_jr_k}\Phi_m, \partial^2_{r_j\theta_k}\Phi_m\\
				\partial^2_{\theta_jr_k}\Phi_m, \partial^2_{\theta_j \theta_k}\Phi_m
			\end{pmatrix}, 1\le j,k\le m.
		\end{equation*}
		Since  $\Phi_m$ is symmetric, a change of variable shows that the Hessian blocks satisfy $H_{j+1,k+1}=H_{jk}$, so there exist $2\times 2$ matrices $C_0, \cdots, C_{m-1}$ such that
		\begin{equation*}
			H_{jk}=C_{k-j(\text{mod}m)}.
		\end{equation*}
		Then $H^*$ is a block-circulant matrix $H^*=\text{Circ}(C_0,\cdots,C_{m-1})$. We remark that $C_l=C_{m-l}^{\text{T}}$ since $H^*$ is real symmetric.
		
To simplify the expressions, in this appendix, we denote
\begin{equation*}
    \rho=r_{\alpha,m}^2= \left(\frac{\beta+1-m}{\beta+1+m}\right)^{\frac{1}{m}},\quad
    D_l=1+\rho^2-2\rho\cos\vartheta_{l+1}, \quad l=1,\cdots,m-1.
\end{equation*}
where $\vartheta_{l+1}=\f{2\pi}{m}l$ and $\beta=\alpha+1$.
Then by straightforward computation, we find that the blocks have the form
		\begin{equation*}
			C_0=\begin{pmatrix}
				a_0&0\\
				0 &c_0\\
			\end{pmatrix},\ 
			C_l=\begin{pmatrix}
				a_l &-b_l\\
				b_l &c_l\\
			\end{pmatrix}, l=1,\cdots, m-1,
		\end{equation*}
		where
		\begin{equation*}
			a_0=-\frac{m^2-1}{3\rho}-\frac{4}{(1-\rho)}-\frac{4m^2\rho^{m-1}}{(1-\rho^{m})^2},
		\end{equation*}
			\begin{equation*}
			c_0=\frac{m^2-1}{3}-\frac{4\rho}{(1-\rho)}+\frac{4m^2\rho^{m}}{(1-\rho^{m})^2},
		\end{equation*}
		and for $l=1,\cdots,m-1$,
        \[
	a_l=\frac{1}{\rho \sin^2\frac{\vartheta_{l+1}}{2}}
	+\frac{2(1+\rho^2)}{\rho D_l}
	-\frac{2(1-\rho^2)^2}{\rho D_l^2},
	\]
	\[
	c_l=-\frac{1}{\sin^2\frac{\vartheta_{l+1}}{2}}
	+\frac{2(1+\rho^2)}{D_l}
	-\frac{2(1-\rho^2)^2}{D_l^2},
	\]
	\[
	b_l=-\frac{4\sqrt{\rho}(1-\rho^2)\sin\vartheta_{l+1}}{D_l^2},
	\]

By using the Discrete Fourier Transform, we can diagonalize the block circle matrix $H^*$.		
Let $\omega=e^{\frac{2\pi i}{m}}$ be the $m$-th unit root. 
Let $F_m$ be the Discrete Fourier Transform matrix
\begin{equation}\label{def Fm}
    F_m=(w^{jk})_{jk}\quad\text{for}~j,k=0,1,\cdots,m-1. 
\end{equation}
and $E_2$ be the unit matrix
    \[E_2=\left(\begin{matrix}
        1 & 0\\
        0 & 1
    \end{matrix}\right).\]
Then we get that $M=(F_m^\star\otimes E_2)H^*(F_m\otimes E_2)$ is a diagonal matrix, where $\otimes$ is the Kronecker product, and $F_m^\star$ is the conjugate transpose matrix of $F_m$. Precisely, we have
		\begin{equation*}
			M=\text{Diag}(M_0,\cdots, M_{m-1}),
		\end{equation*}
		where
		\begin{equation}\label{Mp}
			M_p=\sum_{l=0}^{m-1}C_l\omega^{pl}=\begin{pmatrix}
				\mu_p & -i\gamma_p\\
				i\gamma_p &\nu_p\\
			\end{pmatrix},
		\end{equation}
		and
		\begin{equation}\label{mu nu gamma}
			\mu_p=\sum_{l=0}^{m-1}a_l\omega^{pl},
			i\gamma_p=\sum_{l=0}^{m-1}b_l\omega^{pl},
			\nu_p=\sum_{l=0}^{m-1}c_l\omega^{pl}.
		\end{equation}
		
		To compute $\mu_p, \nu_p, \gamma_p$ it is useful to introduce the following quantities:
		\begin{equation*}
			R_p:=\sum_{l=1}^{m-1}\frac{\omega^{pl}}{\sin^2\frac{\vartheta_{l+1}}{2}},\
			S_p:=\sum_{l=0}^{m-1}\frac{\omega^{pl}}{D_l},\
			T_p:=\sum_{l=0}^{m-1}\frac{\omega^{pl}}{D_l^2}.
		\end{equation*}
	First, a standard identity gives that
	\begin{equation}\label{Rp}
		R_p=\frac{m^2-1}{3}-2p(m-p).
	\end{equation}
	For $S_p$, using	
	$D_l=(1-\rho\omega^l)(1-\rho\omega^{-l})$,
	 we have
	\begin{equation}\label{exp Dl}
		\frac{1}{D_l}=\frac{1}{1-\rho^2}\sum_{n\in\mathbb{Z}}\rho^{|n|}\omega^{nl},
	\end{equation}
	so
	\begin{equation}\label{Sp}
		S_p=\frac{m(\rho^{p}+\rho^{m-p})}{(1-\rho^2)(1-\rho^{m})}.
	\end{equation}
	Finally, squaring \eqref{exp Dl} and after some computation, we can derive that
			\begin{equation}\label{Tp}
			\begin{aligned}
				T_p=&
				\frac{m}{(1-r^4)^3(1-r^{2m})^2}
				\Big[
				r^{2p}\bigl((p+1)-(p-1)r^4+\rho^m(m-p-1-(m-p+1)r^4)\bigr)\\
				&+r^{2(m-p)}\bigl((m-p+1)-(m-p-1)r^4+\rho^m(p-1-(p+1)r^4)\bigr)\Big].
			\end{aligned}
		\end{equation}

Now plugging \eqref{Rp}, \eqref{Sp} and \eqref{Tp} into \eqref{mu nu gamma}, we obtain
	\begin{equation*}
	\begin{aligned}
		\mu_p
		=&
		a_0+\frac{1}{\rho}R_p
		+\frac{2(1+\rho^2)}{\rho}\left(S_p-\frac{1}{(1-\rho)^2}\right)
		-\frac{2(1-\rho^2)^2}{\rho}\left(T_p-\frac{1}{(1-\rho)^4}\right)\\
		=&	-\frac{2}{\rho(1-\rho^m)^2}
		\Bigl(m(1+\rho^p)-p(1-\rho^m)\Bigr)
		\Bigl(m(1+\rho^{m-p})-(m-p)(1-\rho^m)\Bigr) ,
	\end{aligned}
	\end{equation*}

\begin{equation*}
\begin{aligned}
	\nu_p
	=&
	c_0-R_p
	+2(1+\rho^2)\left(S_p-\frac{1}{(1-\rho)^2}\right)
	-2(1-\rho^2)^2\left(T_p-\frac{1}{(1-\rho)^4}\right)\\
	=&\frac{2}{(1-\rho^m)^2}
	\Bigl(p(1-\rho^m)-m(1-\rho^p)\Bigr)
	\Bigl((m-p)(1-\rho^m)-m(1-\rho^{m-p})\Bigr),
\end{aligned}
\end{equation*}
and
	\begin{equation*}
		\begin{aligned}
			\gamma_p
			=&
			2\sqrt{\rho}(1-\rho^2)\bigl(T_{p-1}-T_{p+1}\bigr)\\
			=&	\frac{2m}{(1-\rho^m)^2\sqrt{\rho}}
			\Bigl(
			(p+(m-p)\rho^m)\rho^p-(m-p+p\rho^m)\rho^{m-p}
			\Bigr).
		\end{aligned}
	\end{equation*}

We further introduce
	\[
	X_p:=(\beta+m)\rho^{m-p},
	\qquad
	Y_p:=(\beta+m)\rho^p,
	\]
	and
	\[
	\hat a_p:=\beta-m+2p,
	\qquad
	\hat b_p:=\beta+m-2p.
	\]
	Then after simplification, we get that
	\begin{equation*}
		\mu_p=-\frac{(X_p+\hat a_p)(Y_p+\hat b_p)}{2\rho},\
		\nu_p=\frac{(X_p-\hat a_p)(Y_p-\hat b_p)}{2},\
			\gamma_p=\frac{\hat a_pY_p-\hat b_pX_p}{2\sqrt{\rho}}.
	\end{equation*}
Thus
	\begin{equation}\label{det Mp}
		\det(M_p)=\mu_p\nu_p-\gamma_p^2=-\frac{(X_pY_p-\hat a_p\hat b_p)^2}{4\rho}=-\frac{4p^2(m-p)^2}{\rho}
	\end{equation}
	In particular,
	\begin{equation}\label{M0}
		M_0=\begin{pmatrix}
		-\frac{2(\beta^2-m^2)}{\rho}	&0\\
		0&0\\
		\end{pmatrix}
	\end{equation}
	has a zero eigenvector $(0,1)^{\text{T}}$, and for $p=1,\cdots,m-1$, $M_p$ has two nonzero eigenvalues of opposite signs.

In conclusion, $M$ has a unique $0$ eigenvalue with eigenvector $(0,1,0,\cdots,0)^{\text{T}}.$ Hence the unique $0$-eigenvector (up to a scalar) of $H^*$ is
\begin{equation*}
	(0,1,0,1,\cdots,0,1)^{\text{T}}.
\end{equation*}
Finally, transforming back to $H$ yields the $0$-eigenvector
\begin{equation*}
	\boldsymbol{v}=(0,1,-\sin\vartheta_2,\cos\vartheta_2,\cdots,-\sin\vartheta_{m},\cos\vartheta_{m}).
\end{equation*}
This completes the proof of Lemma \ref{lem evalue H}.

\vs
\noindent {\bf Acknowledgements} \quad 
The research of J. Wei is partially supported by National Key \text{R\&D} Program of China (No. 2022YFA1005602), and Hong Kong General Research Fund “New frontiers in singular limits of nonlinear partial differential equations”. 
The research of Z. Chen is supported by National Key \text{R\&D} Program of China (Grant 2023YFA1010002) and NSFC(No. 1222109). 
The research of H. Li is supported by Beijing Institute of Technology Research Fund Program for Young Scholars.

\end{document}